\documentclass[11pt,letterpaper,reqno]{amsart}
\usepackage[top=2cm, bottom=4.5cm, left=2.5cm, right=2.5cm]{geometry}
\usepackage{graphicx}
\usepackage{tikz}
\usepackage{tikz-cd}
\usepackage{hyperref}
\usepackage{stmaryrd}
\usepackage{quiver}
\usepackage{xcolor}
\usepackage{mathtools}
\usepackage{amsmath,amsthm,amsfonts,amssymb,amscd}
\usepackage{subcaption}

\title{AGM aquariums and elliptic curves over arbitrary finite fields}

\author[J. Kayath]{June Kayath}
\address[J. Kayath]{Department of Mathematics, Massachusetts Institute of Technology, Cambridge, MA, 02139, USA}
\email{kayath@mit.edu}

\author[C. Lane]{Connor Lane}
\address[C. Lane]{Deparment of Mathematics, Rose-Hulman Institute of Technology, Terre-Haute, IN, 47803, USA}
\email{lanecf@rose-hulman.edu}

\author[B. Neifeld]{Ben Neifeld}
\address[B. Neifeld]{Deparment of Mathematics, William {\&} Mary, Williamsburg, VA, 23187, USA}
\email{bmneifeld@wm.edu}

\author[T. Ni]{Tianyu Ni}
\address[T. Ni]{School of Mathematical and Statistical Sciences, Clemson University, Clemson, SC, 29634, USA}
\email{tianyuni1994math@gmail.com}

\author[H. Xue]{Hui Xue}
\address[H. Xue]{School of Mathematical and Statistical Sciences, Clemson University, Clemson, SC, 29634, USA}
\email{huixue@clemson.edu}


\def \QQ{\mathbb{Q}}

\def \ZZ{\mathbb{Z}}
\def \Z{\mathbb{Z}}
\newcommand{\fl}{\mathfrak{l}}
\def \H{\mathbb{H}}

\def \FF{\mathbb{F}}

\def \ker{\text{ker}}

\def \disc{\text{disc}}
\def \Ell{\text{Ell}}
\def\cl{\operatorname{Cl}}
\def \End{\text{End}}

\def \O{\mathcal{O}}

\DeclareMathOperator{\AGM}{AGM}
\renewcommand{\mod}{~\text{mod}~}

\theoremstyle{plain}
\newtheorem{theorem}{Theorem}
\newtheorem{lemma}{Lemma}[section]

\newtheorem{corollary}{Corollary}[section]

\newtheorem{proposition}{Proposition}[section]
\theoremstyle{definition}
\newtheorem{definition}{Definition}[section]
\newtheorem*{remark}{Remark}

\makeatletter
\newtheorem*{rep@theorem}{\rep@title}
\newcommand{\newreptheorem}[2]{%
\newenvironment{rep#1}[1]{%
 \def\rep@title{#2 \ref{##1}}%
 \begin{rep@theorem}}%
 {\end{rep@theorem}}}
\makeatother

\newreptheorem{theorem}{Theorem}
\newreptheorem{lemma}{Lemma}

\hypersetup{
  colorlinks = true,
 linkbordercolor = {white},
 allcolors = {blue},
 }
\begin{document}

\begin{abstract}
In this paper, we define a version of the arithmetic-geometric mean (AGM) function for arbitrary finite fields $\mathbb{F}_q$, and study the resulting AGM graph with points $(a,b) \in \mathbb{F}_q \times \mathbb{F}_q$ and directed edges between points $(a,b)$, $(\frac{a+b}{2},\sqrt{ab})$ and $(a,b)$, $(\frac{a+b}{2},-\sqrt{ab})$. The points in this graph are naturally associated to elliptic curves over $\mathbb{F}_q$ in Legendre normal form, with the AGM function defining a 2-isogeny between the associated curves. We use this correspondence to prove several results on the structure, size, and multiplicity of the connected components in the AGM graph. 
\end{abstract}

\maketitle
\tableofcontents

\section{Introduction}
The \textit{arithmetic-geometric mean} is a function $\AGM_{\mathbb{R}}(a,b)$ over $\mathbb{R}^+$ defined as the limit of the sequence $(a_n,b_n)$, where  $a_0=a$, $b_0 = b$, and $(a_n,b_n)=(  \frac{a_{n-1}+b_{n-1}}{2}, \sqrt{a_{n-1}b_{n-1}})$. We define the AGM map over a generic field $K$ of characteristic not $2$ similarly, as the map that takes pairs $(a,b) \in K^2$ to pairs $(\frac{a+b}{2}, \pm\sqrt{ab})$. 

In their 2022 paper \cite{GOSTJellyfish}, Griffin, Ono, Saikia, and Tsai use the AGM over finite fields $\mathbb{F}_q$ with $q \equiv 3 \pmod 4$ to construct sequences of points $(a,b) \in \FF_q^2$ with edges corresponding to the AGM map. In this case, $-1$ is not a square, so in particular exactly one of the nodes $(\frac{a+b}{2}, \sqrt{ab})$ and $(\frac{a+b}{2}, -\sqrt{ab})$ must be a square in $\mathbb{F}_q$ and it is always possible to pick a ``canonical" value of the square root in a way that leads to further solutions. This results in a visually striking graph whose connected components resemble jellyfish.

They then define an elliptic curve corresponding to each point $(a,b)$, $E_\lambda: y^2 = x(x-1)(x-\lambda)$, where $\lambda = \frac{b^2}{a^2}$, and show that the AGM map between points $(a_0,b_0)$ and $(a_1, b_1)$ exactly corresponds to a 2-isogeny between their corresponding elliptic curves $E_{\lambda_0}$ and $E_{\lambda_1}$, using this to deduce several results about the size and multiplicity of the jellyfish. 

In the $q\equiv 1 \pmod 4$ case, things are significantly more complicated. Since $-1$ is a square, it is very possible that both values of the AGM map on a particular point are well-defined and lead to more values, leaving us with no obvious canonical way of picking a branch. In this paper, we take both branches of the square root, giving us a directed graph $G$ where vertices can have both in-degree and out-degree up to two. The structure of this graph is quite different from the $q\equiv 3 \pmod 4$ case treated in \cite{hypergeometryagm,GOSTJellyfish}; we observe a much larger variety of components, including components that are somewhat similar to jellyfish but also connected components of size 4 that we call ``fish" and strongly connected components that we call ``turtles", described in Section \ref{sec:taxonomy}.

In this paper, we fully classify these components and obtain lower bounds on the number of components. Using the theory of complex multiplication on elliptic curves over finite fields, as well as some results about isogeny volcanoes \cite{isogenyvolcanoes}, we are able to derive results on the size and multiplicity of the jellyfish-like components, and  show that the strongly connected ``turtle" components have vertices that correspond exactly to the supersingular curves over $\mathbb{F}_q$. We also show that the graph exhibits predictable ``growth" as we consider extensions of an AGM graph over $\mathbb{F}_p$ to higher powers  $\mathbb{F}_q, q = p^r$ of $p$.

\begin{definition}\label{def:AGM}
    Let $(a,b)\in \FF_q \times \FF_q$. Define
    \begin{equation*}
        \AGM(a,b)=\begin{cases}
            \varnothing, & ab \in \FF_q^{\times 2} \\
            \left\{\left(\frac{a+b}{2},\sqrt{ab}\right),\left(\frac{a+b}{2},-\sqrt{ab}\right)\right\}, & ab\not\in \FF_q^{\times 2}
        \end{cases}
    \end{equation*}
    where $\sqrt{ab}$ is a square root of $ab$. We write $(a,b) \mapsto (a',b')$ if $(a',b') \in \AGM(a,b)$. 
\end{definition}

\begin{definition}\label{def:AFqVFq}
    Let $q=p^n$ be an odd prime power. The \textit{aquarium} $A(\FF_q)$ associated to $\FF_q$ is a directed graph with vertex set
    \begin{equation*}
        V(\FF_q)=\{(a,b)\in \FF_q^2:a,b\neq 0,a\neq \pm b\}
    \end{equation*}
    where two vertices $(a,b)$ and $(a',b')$ have a directed edge from $(a,b)$ to $(a',b')$ if $(a',b')\in \AGM(a,b)$.
\end{definition}
We remark that for $n\mid m$, the aquarium $A(\FF_{p^n})$ embeds naturally into $A(\FF_{p^m})$. We can use this to prove results on the connected component containing a particular point $P$ in multiple nested aquariums $A(\FF_{p^n})$, as the component ``grows" and gains edges and points (see Section \ref{sec:idealclasses} for more discussion). 
\begin{definition}
    We define the \textit{lambda function} $\lambda: V(\FF_q) \rightarrow \mathbb{F}_q\setminus\{0,1\}$ as $\lambda(a,b) = \frac{b^2}{a^2}$. 
\end{definition}

\begin{definition} \label{def:lambda}
        The \textit{elliptic curve associated to} $ (a,b)\in  V(\FF_q)$ is the curve $E_{\lambda(a,b)} : y^2 = x(x-1)(x-\lambda(a,b))$. 
\end{definition}

In Section \ref{sec:prelims} we discuss some basic properties of elliptic curves. Then in Section \ref{sec:isoellipticcurves} we look in to the properties of pairs $(a,b)$ that define isomorphic elliptic curves. In Section \ref{sec:taxonomy}, we prove some results on the classification and structure of connected components that are required for later sections.

With these baseline result established, we can begin proving our main theorems. in Section \ref{sec:idealclasses}, we prove our first two results, on the multiplicity of jellyfish and on the ``growth" of components in larger graphs.

\begin{theorem}[Generalization of \cite{hypergeometryagm}, Theorem 1.3]\label{jellyfishmultiplicity}
Let $q = p^r$ be an odd prime power, and let $P_0$ be a point on a jellyfish head $H$, associated to an elliptic curve $E_\lambda$. Let $|H|$ be the number of vertices in $H$, and let $M_H$ be the number of scalar multiples of the jellyfish head component in the AGM graph. Then if $\mathcal{O} := \text{End}_{\FF_q}(E_\lambda)$ and $h_2(\mathcal{O})$ denotes the order of $[\mathfrak{p}_2]$ in $cl(\mathcal{O})$, where $\mathfrak{p}_2$ is a prime above $(2)$ in $\mathcal{O}$, we have 

\begin{equation} 
h_2(\mathcal{O}) \mid |H|
\end{equation}
\begin{equation}
M_H \cdot |H| = (q-1) \cdot h_2(\mathcal{O}).
\end{equation}
\end{theorem}

See Theorem \ref{jellyfishmultiplicity} for the precise definition of $M_H$. 

\begin{theorem}\label{fishfate}
    Let $P$ be a point in $V(\mathbb{F}_q)$ for some odd prime power $q$. Then:
    \begin{enumerate}
        \item If $E_{\lambda(P)}$ is supersingular, then the connected component in $A(\mathbb{F}_{q^m})$ containing $P$ is strongly connected, for any $q^m$ square. 
        \item If $E_{\lambda(P)}$ is ordinary and has complex multiplication by an order $\O$ with fraction field $K$:
        \begin{enumerate}
            \item If $\left(\frac{{\rm disc}(K)}{2} \right) = 1$, there is some $n$ such that the connected component containing $P$ is a jellyfish, for all $A(\mathbb{F}_{q^m})$ with $n|m$. 
            \item If $\left(\frac{{\rm disc}(K)}{2} \right) \neq 1$, then the connected component containing $P$ is always acyclic, for any aquarium $A(\mathbb{F}_{q})$. 
        \end{enumerate}
    \end{enumerate}
\end{theorem}

In Section \ref{sec:counting}, we prove our last theorem, which gives an asymptotic lower bound on the number of connected components. It is a generalization of \cite[Theorem 5]{GOSTJellyfish}. 

\begin{theorem} \label{countingcomponents}
    Let $q=p^n$ and let $d(\FF_q)$ denote the number of jellyfish of $A(\FF_q)$. Let $\varepsilon>0$.
    \begin{enumerate}
        \item If $q\equiv 5\mod 8$, then for all sufficiently large $q$
        \begin{equation*}
            d(\FF_q)\geq \left(\frac{p-1}{8p}-\varepsilon\right)\sqrt{q}.
        \end{equation*}
        \item If $q\equiv 1\mod 8$, then for all sufficiently large $q$
        \begin{equation*}
            d(\FF_q)\geq \left(\frac{p-1}{32p}-\varepsilon\right)\sqrt{q}.
        \end{equation*}
    \end{enumerate}
\end{theorem}

\vspace{10pt}
\section{Preliminaries}\label{sec:prelims}

We begin by establishing some results on elliptic curves in Legendre normal form over a field $K$ of characteristic not $2$, $E_{\lambda}:y^2=x(x-1)(x-\lambda)$. Letting $\lambda\in K$ be arbitrary, this parameterizes elliptic curves equipped with two distinguished $2$-torsion points. The $j$-invariant for $E_{\lambda}$ is given by

\begin{equation*}
    j(\lambda)=j(E_{\lambda})=2^8\frac{(\lambda^2-\lambda+1)^3}{\lambda^2(\lambda-1)^2}.
\end{equation*}

This $j$-invariant map is a $6$-to-$1$ ramified covering $X(2)\cong\mathbb{P}^1\to \mathbb{P}^1\cong X(1)$. The fibers of $j(\lambda)$ are

\begin{equation*}
    \Lambda_\lambda=\left\{\lambda,\frac1\lambda,1-\lambda,\frac1{1-\lambda},\frac{\lambda}{\lambda-1},\frac{\lambda-1}{\lambda}\right\}.
\end{equation*}

The map is ramified over $j=0$, $1728$, and $\infty$; see \cite[p.~49]{Silverman1986TheAO}.

Before we proceed with the rest of this section, we must state a concrete version of the $2$-descent lemma, which will be readily used throughout this paper. Its proof can be found in many texts on elliptic curves, for example \cite[Prop X.1.4]{Silverman1986TheAO} or \cite[Prop. I.20]{KoblitzECMF}.

\begin{lemma}[Classical $2$-descent]\label{lem:2descent}
    Let $K$ be a field of $\text{char}(K)\neq 2$ and $E:y^2=(x-\alpha)(x-\beta)(x-\gamma)$ be an elliptic curve with $\alpha,\beta,\gamma\in K$. Then a point $(x,y)\in E(K)$ is also in $2E(K)$ if and only if the three numbers
    \begin{equation*}
        x-\alpha \qquad x-\beta \qquad x-\gamma
    \end{equation*}
    are all squares in $K$.
\end{lemma}

We first develop some lemmas allowing the explicit computations necessary for Section \ref{sec:counting}.

\begin{lemma}\label{lem:doublingformulas}
    Let $K$ be a field of ${\rm char}(K)\neq 2$, $\lambda\in K$, $E_\lambda:y^2=x(x-1)(x-\lambda)$, and $P=(x_0,y_0)\in E(K)$. Then the $x$-coordinate of $2P=(x_1,y_1)$ is given by
    \begin{equation*}
        x_1=\left(\frac{3x_0^2-2(1+\lambda)x_0+\lambda}{2y_0}\right)^2+1+\lambda-2x_0.
    \end{equation*}
\end{lemma}

\begin{proof}
    Expanding $E_\lambda:y^2=x(x-1)(x-\lambda)$, we obtain $y^2=x^3-(1+\lambda)x^2+\lambda x$. By implicit differentiation, we see that the tangent line to $(x_0,y_0)$ is given by
    \begin{equation*}
        L:y=\frac{3x_0^2-2(1+\lambda)x_0+\lambda}{2y_0}(x-x_0)+y_0.
    \end{equation*}
    Substituting this into $E_\lambda$, we get
    \begin{equation*}
        \left(\frac{3x_0^2-2(1+\lambda)x_0+\lambda}{2y_0}(x-x_0)+y_0\right)^2=x^3-(1+\lambda)x^2+\lambda x.
    \end{equation*}
    This is a degree three polynomial with roots $x_0,x_0,x_1$, so by Vieta's formulas on the coefficient of $x^2$ we get $2x_0+x_1=-(x^2\text{ coefficient})$, or
    \begin{equation*}
        x_1=\left(\frac{3x_0^2-2(1+\lambda)x_0+\lambda}{2y_0}\right)^2+1+\lambda-2x_0,
    \end{equation*}
    as desired.
\end{proof}

Using this lemma, we can explicitly compute the $4$-torsion points of $E_{\lambda}$.

\begin{lemma}\label{lem:order4points}
    The $6$ distinct $x$-coordinates of the points $P\in E(\bar K)$ of exact order $4$ are
    \begin{center}
        \begin{tabular}{|c|c|c|}
            \hline
            $2P$ & $P$ & $P$ \\
            \hline
            $(0,0)$ & $\sqrt\lambda$ & $-\sqrt\lambda$ \\
            \hline
            $(1,0)$ & $1+\sqrt{1-\lambda}$ & $1-\sqrt{1-\lambda}$ \\
            \hline
            $(\lambda,0)$ & $\lambda+\sqrt{\lambda^2-\lambda}$ & $\lambda-\sqrt{\lambda^2-\lambda}$ \\
            \hline
        \end{tabular}
    \end{center}
\end{lemma}

\begin{proof}
    These follow from explicit calculations using Lemma \ref{lem:doublingformulas}. In particular, substituting $x_1=0,1,\lambda$ and $y_0^2=x_0^3-(1+\lambda)x_0^2+\lambda x_0$, we get
    \begin{align*}
        0&=(3x_0^2-2(1+\lambda)x_0+\lambda)^2+4(x_0^3-(1+\lambda)x_0^2+\lambda x_0)(1+\lambda-2x_0), \\
        0&=(3x_0^2-2(1+\lambda)x_0+\lambda)^2+4(x_0^3-(1+\lambda)x_0^2+\lambda x_0)(\lambda-2x_0), \\
        0&=(3x_0^2-2(1+\lambda)x_0+\lambda)^2+4(x_0^3-(1+\lambda)x_0^2+\lambda x_0)(1-2x_0).
    \end{align*}
    Simplifying and solving for $x_0$ yields the result.
\end{proof}

We have the following result, analogous to \cite[Lemma 2]{GOSTJellyfish}.

\begin{lemma}\label{lem:2torsingeneral}
    Let $K$ be a field of odd characteristic and $\alpha\neq-1,0,1\in K$. Then the $2$-Sylow subgroup of $E_{\alpha^2}(\FF_q)$ contains $\ZZ/2\ZZ\oplus \ZZ/4\ZZ$.
\end{lemma}

\begin{proof}
    We do casework based on whether $i\in K$. If $i\not \in K$, then $2$-descent (Lemma \ref{lem:2descent}) on $(1,0)$ and $(\alpha^2,0)$ gives the following two triples $(x-\alpha,x-\beta,x-\gamma)$:
    \[\begin{matrix}
                (1,0): & 1 & 0 & 1-\alpha^2, \\
        (\alpha^2,0): \hspace{6pt}&  \alpha^2& \alpha^2-1 & 0.\\
    \end{matrix} \]
    Since $1,0,\alpha^2$ are squares, we have that $(1,0)\in 2E(K)$ if and only if $1-\alpha^2\in K^{\times 2}$, and $(\alpha^2,0)\in 2E(K)$ if and only if $\alpha^2-1\in K^{\times2}$. But since $i\not\in K$, exactly one of these is square.

    If $i\in K$, then applying $2$-descent on $(0,0)$ yields
    \begin{equation*}
        (0,0): \qquad 0 \qquad -1 \qquad -\alpha^2,
    \end{equation*}
    all of which are squares since $i\in K$.
\end{proof}

\begin{lemma}\label{lem:isooverFq}
    Let $K$ be a field of odd characteristic with $i\in K$ and $\alpha\neq -1,0,1\in K$ such that $1-\alpha^2$ is a square. Let $\Lambda_{\alpha^2}$ be the set
    \begin{equation*}
        \Lambda_{\alpha^2}=\left\{\alpha^2, \frac1{\alpha^2},1-\alpha^2,\frac1{1-\alpha^2}, \frac{\alpha^2}{\alpha^2-1}, \frac{\alpha^2-1}{\alpha^2}\right\},
    \end{equation*}
    (the fiber of $j(\alpha^2)$). Then 
    \begin{enumerate}
        \item $\Lambda_{\alpha^2}\subseteq K^{\times 2}$.
        \item The elliptic curves $E_{\lambda}$ for $\lambda \in \Lambda_{\alpha^2}$ are all isomorphic over $K$. 
    \end{enumerate}
\end{lemma}

\begin{proof}
\noindent\begin{enumerate}

   \item We know that $\alpha^2$ and $1-\alpha^2$ are squares by assumption. This implies that their reciprocals $1/\alpha^2$ and $1/(1-\alpha^2)$ are squares. Then
    \begin{equation*}
        \frac{\alpha^2-1}{\alpha^2}=1-\frac{1}{\alpha^2}\equiv \alpha^2-1\equiv -(\alpha^2-1)\equiv 1-\alpha^2 \mod K^{\times 2},
    \end{equation*}
    (note that $\alpha^2-1\equiv -(\alpha^2-1)$ requires $i\in K$) so $(\alpha^2-1)/\alpha^2$ is also square, which implies its reciprocal is square.

    \item Beginning with $E_{\alpha^2}$, we have the following isomorphisms that can be checked with direct calculation:
    \begin{enumerate}
        \item[(a)] $E_{1/\alpha^2}$: the map $\varphi_{1/\lambda}:E_{\alpha^2}\to E_{1/\alpha^2}$ given by $(x,y)\mapsto (x/\alpha^2,y/\alpha^3)$ is an isomorphism;
        \item[(b)] $E_{1-\alpha^2}$: the map $\varphi_{1-\lambda}:E_{\alpha^2}\to E_{1-\alpha^2}$ given by $(x,y)\mapsto (1-x,iy)$ is an isomorphism.
    \end{enumerate}
    Since the transformations $\alpha^2\mapsto 1-\alpha^2$ and $\alpha^2\mapsto 1/\alpha^2$ generate the set $\Lambda_{\alpha^2}$ and the hypothesis of the lemma holds for all $\lambda\in\Lambda_{\alpha^2}$ by (1), we may repeatedly compose $\varphi_{1/\lambda}$ and $\varphi_{1-\lambda}$ to obtain the result. 
\end{enumerate}

\end{proof}

The following Lemma will be useful in Section \ref{sec:idealclasses} for using the theory of CM of these elliptic curves to study the AGM graph. 

\begin{lemma}\label{lem:rivalscorresponding2tors}
    Let $\lambda\in \bar K$ with ${\rm char}(\bar K)>2$. For the elliptic curve $E_\lambda$, denote the points $(0,0)=P_{0,\lambda},(1,0)=P_{1,\lambda},(\lambda,0)=P_{2,\lambda}$.

    Let $E_\lambda$ be an elliptic curve with $j(E_\lambda)\neq 0,1728$. Let $\lambda'\in \Lambda_\lambda$, then there is a unique permutation $\sigma_{\lambda'}\in S_3=\text{Sym}\{0,1,2\}$ such that for any isomorphism $\varphi:E_\lambda\to E_{\lambda'}$, we have
    \begin{equation*}
        \varphi(P_{0,\lambda})=P_{\sigma_{\lambda'}(0),\lambda'} \qquad \varphi(P_{1,\lambda})=P_{\sigma_{\lambda'}(1),\lambda'} \qquad \varphi(P_{2,\lambda})=P_{\sigma_{\lambda'}(2),\lambda'}
    \end{equation*}
    Moreover, for each $\lambda'\in \Lambda_\lambda$, the associated $\sigma_{\lambda'}$ is given by
    \begin{align*}
        &\sigma_{\lambda}=\text{id} &&\sigma_{1/\lambda}=(12) &&&\sigma_{1-\lambda}=(01) \\
        &\sigma_{1/(1-\lambda)}=(021) &&\sigma_{(\lambda-1)/\lambda}=(012) &&&\sigma_{\lambda/(\lambda-1)}=(02).
    \end{align*}
\end{lemma}

\begin{proof}
    First we show uniqueness of $\sigma_{\lambda'}$. Suppose $\varphi_1,\varphi_2:E_{\lambda}\to E_{\lambda'}$ are two isomorphisms of elliptic curves. Define $\sigma_i$ for $i\in\{0,1\}$ by
    \begin{equation*}
        \varphi_i(P_{j,\lambda})=P_{\sigma_i(j),\lambda'}
    \end{equation*}
    for $j\in\{0,1,2\}$. Then $\psi=\varphi_2^{-1}\circ \varphi_1:E_{\lambda}\to E_{\lambda}$ is an automorphism of $E_{\lambda}$. We have $\psi(P_{j,\lambda})=P_{\sigma_2^{-1}\sigma_1(j),\lambda}$. Since $j(E_{\lambda})\neq0,1728$, all of its automorphisms fix $E[2]$, so $\sigma_2^{-1}\sigma_1=\text{id}$ and we obtain uniqueness.

    Now we compute $\sigma_{\lambda'}$ for each $\lambda'\in \Lambda_\lambda$, for this we use the isomorphisms given in Lemma \ref{lem:isooverFq}. We first verify $\sigma_{1/\lambda}=(12)$ and $\sigma_{1-\lambda}=(01)$:
    \begin{align*}
        &\varphi_{1/\lambda}(0,0)=(0,0), &&\varphi_{1/\lambda}(1,0)=(1/\lambda,0),  &&&\varphi_{1/\lambda}(\lambda,0)=(1,0), \\
        &\varphi_{1-\lambda}(0,0)=(1,0), &&\varphi_{1-\lambda}(1,0)=(0,0), &&&\varphi_{1-\lambda}(\lambda,0)=(1-\lambda,0).
    \end{align*}
    From here, we use the uniqueness of $\sigma_{\lambda'}$ to determine the rest
    \begin{align*}
        \sigma_{1/(1-\lambda)}&=\sigma_{1/\lambda}\sigma_{1-\lambda}=(021), \\
        \sigma_{(\lambda-1)/\lambda}&=\sigma_{1-\lambda}\sigma_{1/\lambda}=(012), \\
        \sigma_{\lambda/(\lambda-1)}&=\sigma_{1/\lambda}\sigma_{1-\lambda}\sigma_{1/\lambda}=(02).
    \end{align*}
    The proof is complete.
\end{proof}

Lastly, we note the following result about supersingular elliptic curves in Legendre normal form, which will be required for some of our results in Section \ref{sec:idealclasses}. 
\begin{lemma}[\cite{fourthpowers}, Proposition 3.1]\label{lem:4thpowers} 
Let $E$ be a supersingular elliptic curve over $\overline{\mathbb{F}}_p$ for an odd prime $p$. Then $E$ is expressible in Legendre normal form $y^2 = x(x-1)(x-\lambda)$ where $\lambda \in \mathbb{F}_{p^2}^{\times 4}$ is a fourth power over $\mathbb{F}_{p^2}$.
\end{lemma}

With some general results about elliptic curves in Legendre normal form established, we now want to look at AGM and its relationship with elliptic curves. For the following, we recall Definition \ref{def:AGM}:

\begin{equation*}
    \AGM(a,b)=\begin{cases}
        \varnothing & ab \in \FF_q^{\times 2}, \\
        \left\{\left(\frac{a+b}{2},\sqrt{ab}\right),\left(\frac{a+b}{2},-\sqrt{ab}\right)\right\} & ab\not\in \FF_q^{\times 2}.
    \end{cases}
\end{equation*}

Also recall Definition \ref{def:AFqVFq} of $V(\FF_q)$ and $A(\FF_q)$.

\begin{lemma}[c.f. {\cite[Thm. 3]{GOSTJellyfish}}]\label{lem:theisog}
    The following hold.
    \begin{enumerate}
        \item For all $\alpha\neq -1,0,1\in \FF_q$, $\alpha^2$ has exactly $2(q-1)$ preimages in $V(\FF_q)$ under the map $(a,b)\mapsto {b^2/a^2}$.
        \item Let $K$ be a field with $char{K}\neq 2$ and let $(a',b')\in \AGM(a,b)$. Then there exists an isogeny
        \begin{equation*}
            \varphi_{a,b}:E_{\lambda(a,b)}\to E_{\lambda(a',b')}
        \end{equation*}
        with kernel $\langle (0,0)\rangle$.
        \item The isogeny dual to $\varphi_{a,b}$ has kernel $\langle (1,0)\rangle$. 
    \end{enumerate}
\end{lemma}

\begin{proof} \noindent
\begin{enumerate}
    \item The equation $b^2/a^2=\alpha^2$ splits in to $b/a=\alpha$ and $b/a=-\alpha$, both of which have $q-1$ solutions. 
    \item Define the isogeny $\varphi=\varphi_{a,b}:E_{\lambda(a,b)}\to E_{\lambda(a',b')}$ by
    \begin{equation*}
        \varphi(x,y)=\left(\frac{(ax+b)^2}{x(a+b)^2},-\frac{ay(ax-b)(ax+b)}{x^2(a+b)^3}\right).
    \end{equation*}
    By direct computation we see that the image of $\varphi$ lands in $E_{\lambda(a',b')}$:
    \begin{equation*}
        \left(-\frac{ay(ax-b)(ax+b)}{x^2(a+b)^3}\right)^2=\left(\frac{(ax+b)^2}{x(a+b)^2}\right)\left(\frac{(ax+b)^2}{x(a+b)^2}-1\right)\left(\frac{(ax+b)^2}{x(a+b)^2}-\frac{4ab}{(a+b)^2}\right).
    \end{equation*}
    By homogenizing, we see that $\varphi(O)=O$, so $\varphi_{a,b}$ is an isogeny. Finally, by inspection we deduce that $\ker (\varphi)=\langle (0,0)\rangle$. 

    \item Since $\hat\varphi\circ \varphi=[2]$, where $\hat\varphi$ is dual to $\varphi$, we know that $\ker (\hat \varphi)=\varphi(E[2])$. We know that $\varphi(0,0)=\varphi(O)=O$, so it suffices to determine whether $\varphi(1,0)=\varphi(\lambda,0)=(1,0)$, or $\varphi(1,0)=\varphi(\lambda,0)=(\lambda',0)$. Indeed by direct computation we verify that
    \begin{equation*}
        \varphi(1,0)=\left(\frac{(a\cdot 1+b)^2}{1\cdot (a+b)^2},-\frac{0\cdot(a\cdot 1-b)(a\cdot 1+b)}{1\cdot (a+b)^3}\right)=(1,0),
    \end{equation*}
    which forces $\ker(\hat\varphi)=\langle (1,0)\rangle$. 
    \end{enumerate}
    As desired.
\end{proof}

Recall the notation $(a,b)\mapsto (a',b')$ given in Definition \ref{def:AGM} to mean $(a',b')\in \AGM(a,b)$.

\begin{corollary}\label{cor:nomultby2}
    Let $(a_0,b_0)\mapsto(a_1,b_1)\mapsto(a_2,b_2)$. Then $\ker(\varphi_{a_1,b_1}\circ\varphi_{a_0,b_0})\neq E_{\lambda(0,0)}[2]$.
\end{corollary}

\begin{proof}
    If this were the case, then $\varphi_{a_1,b_1}=\pm \hat\varphi_{a_0,b_0}$, but this is not true, since $\ker (\varphi_{a_1,b_1})=\langle (0,0)\rangle\neq\langle(1,0)\rangle=\ker( \hat\varphi_{a_0,b_0})$.
\end{proof}

It will be useful in Section \ref{sec:taxonomy} to obtain information on how to go ``backwards" in the AGM graph, and to this end we introduce the following lemma.

\begin{lemma}\label{lem:edgecriterion}
    Let $(a,b),(a',b')\in V(\FF_q)$. Then there is an edge from $(a,b)$ to $(a',b')$ if and only if $(x,y,x',y')=(a,b,a',b')$ is a solution to the system of polynomial equations:
    \begin{equation*}
        \begin{cases}
            x^2-2xx'+(y')^2=0, \\
            y^2-2yx'+(y')^2=0.
        \end{cases}
    \end{equation*}
\end{lemma}

\begin{proof}
    We have $a+b=2a'$ and $ab=(b')^2$, so the result follows by Vieta's formulas.
\end{proof}

\begin{corollary}\label{cor:structureofpreim}
    Let $(a',b')\in V(\FF_q)$. Then there exists $(a,b)$ with $(a',b')\in \AGM(a,b)$ if and only if $(a')^2-(b')^2\in \FF_q^{\times 2}$. If this is the case, then
    \begin{equation*}
        a,b=a'+\sqrt{(a')^2-(b')^2},a'-\sqrt{(a')^2-(b')^2}
    \end{equation*}
    in some order. Equivalently, $(a,b,a',b')$ solves the polynomial equation from Lemma \ref{lem:edgecriterion}.
\end{corollary}

\begin{proof}
    By Lemma \ref{lem:edgecriterion}, $a,b$ must be roots of the polynomial
    \begin{equation*}
        x^2-2(a')x+(b')^2=0.
    \end{equation*}
    By the quadratic formula, we obtain
    \begin{equation*}
        a,b=\frac{2a'\pm\sqrt{(-2a')^2-4(b')^2}}{2}=a'\pm\sqrt{(a')^2-(b')^2},
    \end{equation*}
    which proves both claims.
\end{proof}
We will also briefly recall some facts on isogenies between elliptic curves with complex multiplication. 
\begin{lemma}[\cite{isogenyvolcanoes}, Sections 2.7 and 2.9] \label{lem:cmtheory}

Let $\ell \neq p$ be a prime, and let $E_1$ and $E_2$ be elliptic curves over $\mathbb{F}_q$, where $q = p^r$. Now let $\varphi\colon E_1\to E_2$ by an $\ell$-isogeny (isogeny of degree $\ell$) of elliptic curves with CM by imaginary quadratic orders $\O_1$ and $\O_2$ respectively (i.e. $\End_{\bar{\FF}_q}(E)=\O_1$). Then $\O_1=\Z+\tau_1\Z$ and $\O_2=\Z+\tau_2\Z$, for some $\tau_1,\tau_2\in\H$.
The isogeny $\hat\varphi\circ\tau_2\circ\varphi$ lies in $\End_{\overline{\mathbb{F}}_q}(E_1)$, and this implies that $\ell\tau_2\in\O_1$;
similarly, $\ell\tau_1\in\O_2$.
There are thus three possibilities:
\begin{enumerate}
\item $\O_1=\O_2$, in which case $\varphi$ is called \emph{horizontal};
\item $[\O_1:\O_2] = \ell$, in which case $\varphi$ is called \emph{descending};
\item $[\O_2:\O_1] = \ell$, in which case $\varphi$ is called \emph{ascending}.
\end{enumerate}
The orders $\O_1$ and $\O_2$ necessarily have the same fraction field $K=\End^0(E_1)=\End^0(E_2)$, and both lie in the maximal order $\O_K$, the ring of integers of $K$.

Every horizontal $\ell$-isogeny $\varphi$ arises from the action of an invertible $\O$-ideal $\fl$ of norm $\ell$, namely, the ideal of endomorphisms $\alpha\in\O$ whose kernels contain the kernel of $\varphi$.
If $\ell$ divides the index of $\O$ in the maximal order~$\O_K$ of its fraction field $K$, then no such ideals exist.
Otherwise we say that~$\O$ is \emph{maximal at} $\ell$, and there are then exactly
\[
1 + \left(\frac{\mathrm{disc}(K)}{\ell}\right) = \begin{cases}
0\qquad\text{$\ell$ is inert in $K$},\\
1\qquad\text{$\ell$ is ramified in $K$},\\
2\qquad\text{$\ell$ splits in $K$},
\end{cases}
\]
invertible $\O$-ideals of norm $\ell$, each of which gives rise to a horizontal $\ell$-isogeny.
In the split case we have $(\ell)=\fl\cdot\bar\fl$, and the $\fl$-orbits partition $\Ell_\O(k)$ into cycles corresponding to the cosets of $\langle[\fl]\rangle$ in the class group $\cl(\O)$.
\end{lemma}

The following lemma is necessary for the result in Proposition \ref{prop:cycles}. 

\begin{lemma}[\cite{isogenyvolcanoes}, Lemma 6]\label{lem:ascending}
Let $E'/k$ be an elliptic curve with CM by $\O'$, where $\O'$ is the unique order of index $\ell$ in an order $\O$. Then there is a unique ascending $\ell$-isogeny from $E'$ to an elliptic curve $E/k$ with CM by $\O$. 
\end{lemma}
\begin{remark}
    In the case where $\O=\Z[i]$ or $\O=\Z[\zeta_3]$, there is only one elliptic curve $E/k$ with complex multiplication by $\O$ (and it has $j$-invariant 1728 or 0 respectively). Then there are multiple descending isogenies from $E$ to $E'$, but only one ascending isogeny, so the Lemma still holds.
\end{remark}

\begin{lemma}\label{lem:2splitsinendomorphismring}
    Let $q\equiv 1\mod 8$, $\sqrt{q}$ a $2$-adic square root of $q$, and $E/\FF_q$ be an ordinary elliptic curve with trace of Frobenius $s\equiv \pm46\sqrt{q}\pmod{128}$ and endomorphism ring $\mathcal{O}$. Then $2$ splits in $K=\text{Frac}(\mathcal O)$.
\end{lemma}

\begin{proof}
    Let $\pi$ be the Frobenius endomorphism on $E$. We know that $K=\QQ(\pi)$, so $2$ splits in $K$ if and only if $\QQ_2(\pi)$ is the trivial extension of $\QQ_2$. As $\pi$ has minimal polynomial $\pi^2-s\pi+q$, which has discriminant $D_\pi=s^2-4q$, $2$ splits in $K$ if and only if $\QQ_2(\sqrt{D_\pi})=\QQ_2(\pi)$ is the trivial extension of $\QQ_2$, or if $s^2-4q$ is a square in $\QQ_2$.

    By direct computation, we have
    \begin{align*}
        s^2-4q &\equiv (\pm46\sqrt{q})^2-4q\equiv 68q-4q\equiv 64q\pmod{512}.
    \end{align*}
    This implies $D_\pi=s^2-4q=8^2u$ where $u\equiv 1\mod 8$, so $D_\pi$ is a square in $\QQ_2$.
\end{proof}

\section{Pairs That Define Isomorphic Elliptic Curves}\label{sec:isoellipticcurves}

Given a pair $(a,b)\in V(\FF_q)$, an interesting question to ask is: what can we say about the combinatorics of other pairs $(c,d)\in V(\FF_q)$ such that $(a,b)$ and $(c,d)$ define isomorphic elliptic curves (over $\bar{\FF}_q$, via Definition \ref{def:lambda})? Let $\lambda=\lambda(a,b)$. Then $(c,d)$ defining an isomorphic elliptic curve is equivalent to
\begin{equation*}
    \lambda(c,d)\in \Lambda_\lambda.
\end{equation*}
These results are useful for the succeeding sections, as they provide a more precise understanding of the structure of the AGM graph. Using results like Lemmas \ref{lem:isooverFq} and \ref{lem:rivalscorresponding2tors}, we will obtain useful results about the structure of the connected component containing $(a,b)$.

\begin{lemma}\label{lem:cousins}
    Let $(a_0,b_0)$ be a pair with children $(a_0,b_0)\mapsto (a_1,b_1)\mapsto (a_2,b_2)$ and $(a_0,b_0)\mapsto (a_1',b_1')\mapsto (a_2',b_2')$ with $(a_1,b_1)\neq (a_1',b_1')$, under the AGM map. Then
    \begin{equation*}
        \lambda(a_2,b_2)=\frac{\lambda(a_2',b_2')}{\lambda(a_2',b_2')-1}.
    \end{equation*}
\end{lemma}

\begin{proof}
    We introduce the notation $(a_1,b_1)=(a,b)$ and $(a_1',b_1')=(a,-b)$. Then
    \begin{equation*}
        (a_2,b_2)=\left(\frac{a+b}2,\pm\sqrt{ab}\right), \qquad {\rm and}\qquad(a_2',b_2')=\left(\frac{a-b}{2},\pm i\sqrt{ab}\right).
    \end{equation*}
    We compute $\lambda=\lambda(a_2,b_2)=\frac{4ab}{(a+b)^2}$ and $\lambda'=\lambda(a_2',b_2')=\frac{-4ab}{(a-b)^2}$. Then by direct computation we obtain
    \begin{equation*}
        \frac1\lambda+\frac1{\lambda'}=1,
    \end{equation*}
    implying that $\lambda=\lambda'/(\lambda'-1)$.
\end{proof}

\begin{lemma}\label{rivalsexistence}
    Suppose $q\equiv 1\pmod 4$ and let $(a,b)\in V(\FF_q)$ be a pair that has a parent. Let $\lambda=\lambda(a,b)$ and let $\lambda'\in \Lambda_\lambda$. Then there exist pairs $(c,d)\in V(\FF_q)$ with $\lambda(c,d)=\lambda'$, and moreover these $(c,d)$ have parents. 
\end{lemma}

\begin{proof}
    This is essentially a restatement of Lemma \ref{lem:isooverFq} (1). First let $\alpha=b/a$. if $(a,b)$ has a parent, then by Corollary \ref{cor:structureofpreim}, $a^2-b^2$ is a square, and in particular $(a^2-b^2)/a^2=1-\alpha^2$ is a square, which implies all elements $\lambda'\in \Lambda_\lambda$ are square. The pair $(1,\sqrt{\lambda'})\in V(\FF_q)$ then has $\lambda(1,\sqrt{\lambda'})=\lambda'$.

    Now to show that such $(c,d)$ have parents, note that $\lambda(c,d)\in \Lambda_\lambda$, and for any $\lambda'\in \Lambda$ we know that $1-\lambda'\in \Lambda_\lambda$. This implies $1-\lambda(c,d)=1-d^2/c^2=(c^2-d^2)/c^2$ is a square, so $(c,d)$ has a parent by Corollary \ref{cor:structureofpreim}.
\end{proof}

The following lemma will be used in Section \ref{sec:idealclasses} as a crucial step in classifying the ``limiting behavior" of the component containing a pair $(a,b)$ as we pass to larger extensions of $\FF_q$.

We now discuss a Lemma that considers the inclusion $A(\FF_q)\subseteq A(\FF_{q^2})$, which will be necessary for \ref{lem:rivalsconnected}.

\begin{lemma}\label{lem:growthbyonestage}
    Let $(a,b)\in V(\FF_q)$, then $(a,b)$ has both a parent and a child in $V(\FF_{q^2})$.
\end{lemma}

\begin{proof}
    We know that $(a,b)$ has a child (resp. parent) in $V(\FF_{q^2})$ if $ab$ (resp. $\sqrt{a^2-b^2}$) is a square. But every element of $\FF_q$ is a square in $\FF_{q^2}$, so $ab$ (resp. $\sqrt{a^2-b^2}$) is a square in $\FF_{q^2}$ and we obtain the result. 
\end{proof}

\begin{lemma}\label{lem:rivalsconnected}
    Let $(a,b)\in V(\FF_q)$, and $\lambda=\lambda(a,b)$. Then for all $\lambda' \in \Lambda_\lambda$, there exist $(c,d)\in V(\FF_{q^2})$ such that $\lambda(c,d) = \lambda'$, and $(a,b)$ and $(c,d)$ are connected (as an undirected graph) in $A(\FF_{q^4})$.
\end{lemma}

\begin{proof}\noindent We show the lemma in the following steps.
    \begin{enumerate}
    \item We first consider $\lambda'=1/\lambda$, which is obtained by $(c,d)=(b,a)$. These share the same children, so they are connected in $A(\FF_{q^2})$ by \ref{lem:growthbyonestage}.

    \item Next we consider $\lambda'=\lambda/(\lambda-1)$, and such a pair $(c,d)$ is connected to $(a,b)$ by a grandparent $(a'',b'')$ in in $A(\FF_{q^4})$ by Lemmas \ref{lem:cousins} and \ref{lem:growthbyonestage}. We note that this pair $(c,d)$ is defined over $\FF_{q^2}$: indeed, the parents of $(a,b)$, $(a',b')$ are defined over $\FF_{q^2}$, so the parents of this $(c,d)$ (which are $(a',-b')$) are defined over $\FF_{q^2}$. Then $(c,d)$ are defined over $\FF_{q^2}$ if $\sqrt{-a'b'}\in \FF_{q^2}$, which is true because $i\in \FF_{q^2}$ and $\sqrt{ab}\in \FF_{q^2}$.

    \begin{center}
        \begin{tikzcd}
        	{(a'',b'')\in V(\FF_{q^4})} \\
        	{(a',-b')\in V(\FF_{q^2})} & {(a',b')\in V(\FF_{q^2})} \\
        	{(c,d)\in V(\FF_{q^2})} & {(a,b)\in V(\FF_{q})}
        	\arrow[maps to, from=1-1, to=2-1]
        	\arrow[maps to, from=1-1, to=2-2]
        	\arrow[maps to, from=2-1, to=3-1]
        	\arrow[maps to, from=2-2, to=3-2]
        \end{tikzcd}
    \end{center}

    \item Now we consider $\lambda'=(\lambda-1)/\lambda$, which exists over $\FF_{q^2}$ and is connected to $(a,b)$ over $\FF_{q^4}$ by applying the argument in (1) to the pair in (2). (Note that we need that the pair in (2) is defined over $\FF_{q^2}$ to guarantee that it will be connected to this pair in $\FF_{q^4}$.)

    \item For $\lambda'=\frac1{1-\lambda}$, we apply the argument in (2) to the pair in (1). Since the pair in (1) is defined over $\FF_{q}$, this gives us a pair $(c,d)$ defined in $\FF_{q^2}$ that is connected to $(a,b)$ in $\FF_{q^4}$.

    \item For $\lambda'=1-\lambda$, we apply the argument in (1) to pair in (4). This follows similarly to (3), where we get a pair defined over $\FF_{q^2}$ that is connected to $(a,b)$ in $\FF_{q^4}$.
    \end{enumerate}
    The proof is now complete.
\end{proof}

\begin{definition}
    Let $P_0\in V(\mathbb{F}_q)$ be a point. Then the set $R(P_0) \subset V(\mathbb{F}_{q^2})$ is defined as a set of 24 points (or less in the case $j(E)=0,1728$) such that
    \begin{enumerate}
        \item For each  $\lambda' \in \Lambda_{\lambda(a,b)}$, there are exactly four points $P, P'$ that satisfy $\lambda(P) = \lambda(P') = \lambda'$.
        \item For every point $P = (a,b) \in R(P_0)$, the points $(a, -b), (-a, b),(-a, -b)$ and 
        
        \noindent $(b,a), (b,-a), (-b, a), (-b, -a)$ are also in $R(P_0)$. 
        \item $P_0$ is connected to every point in $R(P_0)$ over $A(\mathbb{F}_{q^4})$.
    \end{enumerate}
    
Note that by Lemma \ref{lem:rivalsconnected}, this set always exists. When $j(E)=0,1728$, some of these points may repeat since $|\Lambda(\lambda)|<6$, but we still get a set of connected vertices corresponding to all $\lambda'\in \Lambda(\lambda)$. 
\end{definition}

\begin{lemma}\label{lem:sqrt1minlambda}
    Let $q\equiv 1\mod 4$ and $(a,b)\in V(\FF_q)$, and let $(c,d)\in V(\FF_q)$ be such that $\lambda(c,d)=1-\lambda(a,b)$. Then $(a,b)$ has a grandparent if and only if $(c,d)$ has a child.
\end{lemma}

\begin{proof}
    Note that $(a,b)$ has a parent by Lemma \ref{rivalsexistence}. Its parents is given by $(a_0,b_0)$, where
    \begin{equation*}
        a_0,b_0=a\pm\sqrt{a^2-b^2}=a\pm a\sqrt{1-b^2/a^2}=a\pm a\sqrt{1-\lambda(c,d)}=a\pm a\frac{d}{c}.
    \end{equation*}
    Then by Corollary \ref{cor:structureofpreim}, $(a_0,b_0)$ has a parent if and only if $a_0^2-b_0^2$ is a square. We compute
    \begin{equation*}
        a_0^2-b_0^2=-4a^2\frac{d}{c}\equiv cd\mod \FF_{q}^{\times2},
    \end{equation*}
    which is a square if and only if $(c,d)$ has a child. 
\end{proof}

\section{Taxonomy}\label{sec:taxonomy}
In this section, we develop some results on the structures of connected components in an aquarium. We have three main cases: $q = 3 \pmod 4$, $q = 5 \pmod 8$, and $q = 1 \pmod 8$. The $q \equiv 3 \pmod 4$ case was discussed in \cite{GOSTJellyfish}, but before we can reproduce their result we need to introduce some terminology.

\begin{figure}[h]
    \centering
    \includegraphics[width=0.8\linewidth]{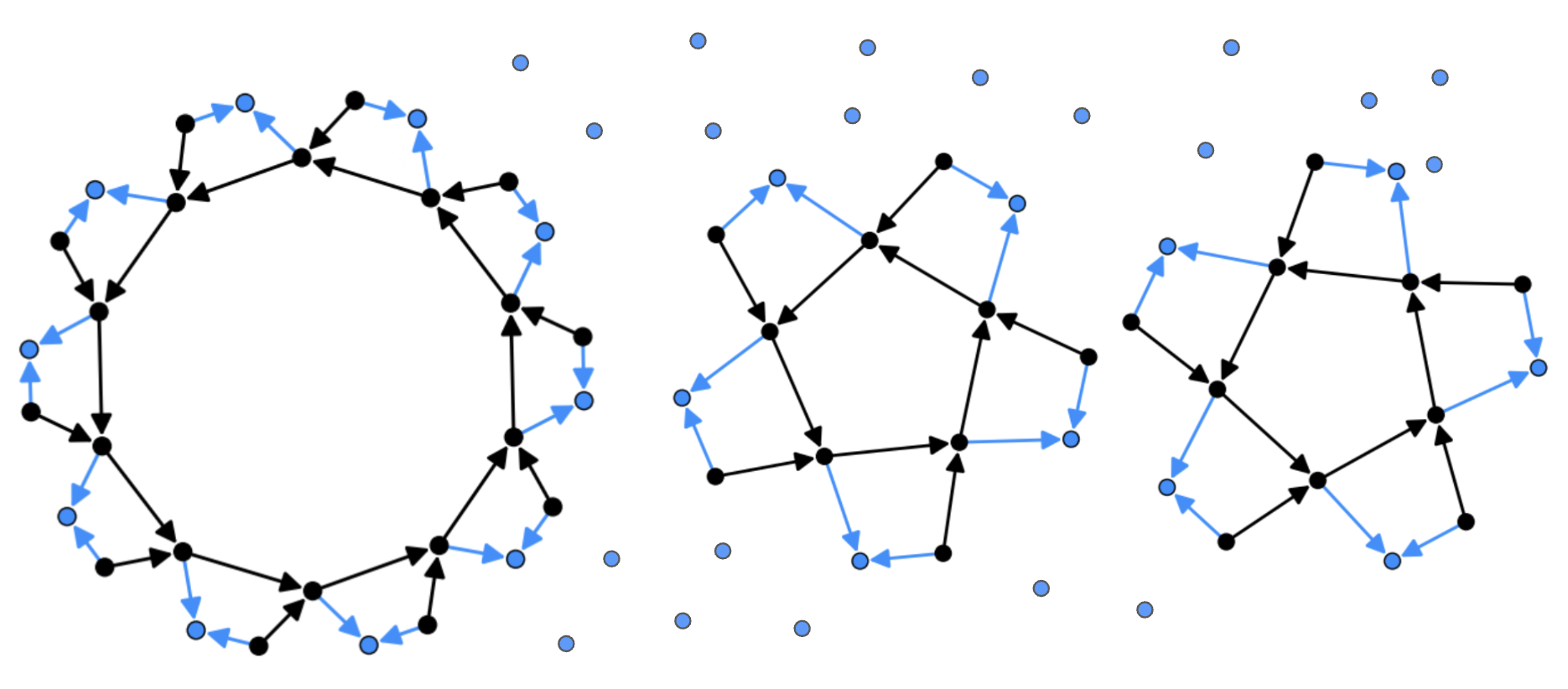}
    \caption{The aquarium $A(\mathbb{F}_{11})$. Here the lighter blue points and edges denote the ``dead end" branches of the square root, so that the black part of the graph is the jellyfish swarm  for $\mathbb{F}_{11}$.}
    \label{fig:enter-label}
\end{figure}

\begin{definition}
    An undirected graph $G$ is \textit{connected} if for every distinct pair of vertices $u, v \in G$, the graph contains a path from $u$ to $v$. A directed graph $G'$ is connected if the undirected graph formed by replacing every directed edge with an undirected edge is connected. 
\end{definition}
\begin{definition}
A \textit{connected component} in a directed graph is a connected subgraph that is not contained in any larger connected subgraph. 
    
\end{definition}

\begin{definition}
    A directed graph $G$ is \textit{strongly connected} if for every distinct pair of vertices $u, v \in G$, there is a (directed) path from $u$ to $v$ and vice versa. Strong connectedness is strictly stronger than connectedness. 
\end{definition}

\begin{definition}
    A \textit{strongly connected component} of a graph $G$ is a strongly connected subgraph that is not contained in any larger strongly connected subgraph. It is \textit{nontrivial} if it has more than one vertex. 
\end{definition}

\begin{definition}\label{def:fish}
    A \textit{fish} is a connected component of size 4. A fish will always contain a point $(a,b)$, another point $(b,a)$, and their shared children. Every edge in a fish is associated to the same isogeny.

    \[\begin{tikzcd}
	{(a,b)} & {(b,a)} \\
	{(\frac{a+b}{2}, \sqrt{ab})} & {(\frac{a+b}{2}, -\sqrt{ab})}
	\arrow[from=1-1, to=2-1]
	\arrow[from=1-1, to=2-2]
	\arrow[from=1-2, to=2-1]
	\arrow[from=1-2, to=2-2]
\end{tikzcd}\]
This structure is the basic building block of connected components in the aquarium, because if a component contains any edge in this structure it must contain the other three. 
\end{definition}

\begin{definition}
    A \textit{jellyfish} is a connected component in a graph $A(\mathbb{F}_q)$ that is not itself strongly connected but contains at least one nontrivial strongly connected component. 
\end{definition}

\begin{remark}\noindent
    \begin{enumerate}
        \item We will show in Proposition \ref{prop:cycles} that the strongly connected components in a jellyfish are always simple cycles, called jellyfish heads or bell heads. 
        \item In \cite{GOSTJellyfish}, the authors did not give a precise definition of what a jellyfish was, outside of referring to the connected components of $A(\FF_q)$, $q\equiv 3\mod 4$ as jellyfish. For our purposes, it is useful to have a precise graph-theoretic definition of jellyfish. 
    \end{enumerate}
\end{remark}

Now, we can state the $q\equiv 3\mod 4$ case of the classification of connected components given in \cite{GOSTJellyfish}.

\begin{proposition}[\cite{GOSTJellyfish}, Theorem 1(3)]
If $\mathbb{F}_q$ is a finite field with $q \equiv 3 \pmod 4$, then every nontrivial connected component in $A(\mathbb{F}_q)$ is a jellyfish, and every jellyfish has a bell head with length one tentacles pointing to each node.
\end{proposition}
While our statement of this theorem is different from \cite{GOSTJellyfish}, the proof still applies. However, because we are taking both branches of the square root, our ``jellyfish swarms" in the $q = 3 \pmod 4$ case look somewhat different.

\begin{figure}[h!]
    \centering
    \includegraphics[width=0.8\linewidth]{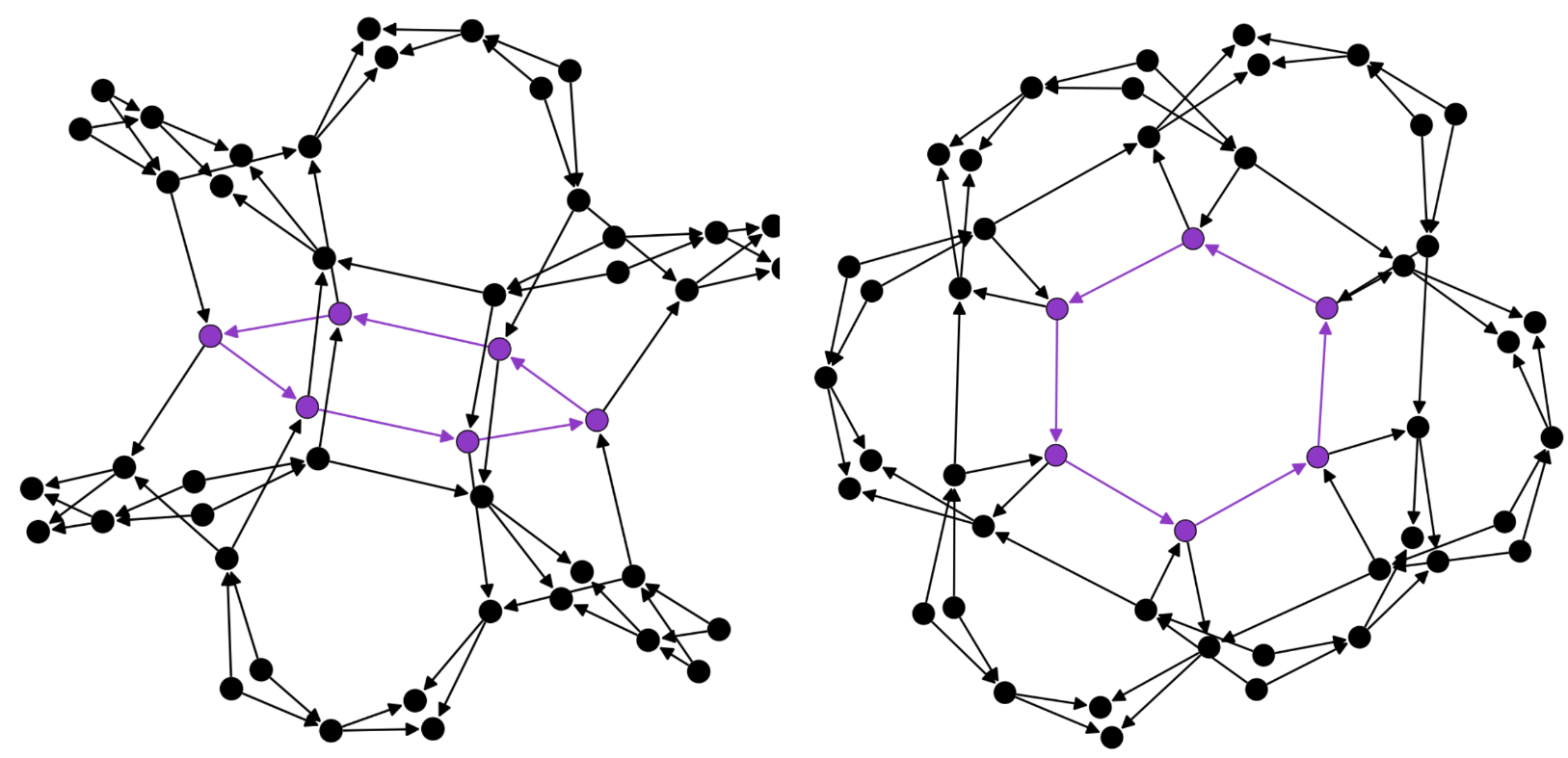}
    \caption{Two views of a jellyfish with 48 total points and a strongly connected component of size 6 in the center, highlighted in purple. This jellyfish arises in the graph $A(\mathbb{F}_{125})$. }
    \label{fig:48fish}
\end{figure}

We will also show that isolated points, fish, and jellyfish are the only connected components that can arise in the $q \equiv 5 \pmod 8$ case.
 
\begin{lemma}\label{lem:fish}
    Let $C$ be a connected component of $A(\FF_q)$, and suppose $C$ contains the edge $(a,b) \mapsto (\frac{a+b}{2}, \sqrt{ab})$. Then $C$ contains the points $(b,a)$ and $(\frac{a+b}{2}, -\sqrt{ab})$ also, in a ``fish" configuration.
\end{lemma}
\begin{proof}
    Clearly $(\frac{a+b}{2}, -\sqrt{ab}) \in \AGM(a,b)$, and $\AGM(a,b) = \AGM(b,a)$. 
\end{proof}

\begin{lemma}\label{lem:5mod8}
    Let $q\equiv 5 \mod 8$. Let $(a_0,b_0)\in V(\FF_q)$ be such that there exists $(a_{-1},b_{-1})\mapsto (a_0,b_0)\mapsto (a_1,b_1)$.
    \begin{enumerate}
        \item There exists $(a_0',b_0')\in V(\FF_q)$ with $\AGM(a_0',b_0')=\AGM(a_0,b_0)$, and moreover $a_0'=b_0$ and $b_0'=a_0$.
        \item There exists $(a_{-1}',b_{-1}')\in V(\FF_q)$ such that $(a_{-1}',b_{-1}')\mapsto (a_0',b_0')$.
        \item Exactly one of the following is true
        \begin{enumerate}
            \item There exists $(a_{-2},b_{-2})\in V(\FF_q)$ such that $(a_{-2},b_{-2})\mapsto (a_{-1},b_{-1})$.
            \item There exists $(a_{-2}',b_{-2}')\in V(\FF_q)$ such that $(a_{-2}',b_{-2}')\mapsto (a_{-1}',b_{-1}')$.
        \end{enumerate}
    \end{enumerate}
    \begin{tikzcd}
    	{\text{Assuming}} && {3(a)} &&& {3(b)} \\
    	&& {(a_{-2},b_{-2})} &&&& {(a_{-2}',b_{-2}')} \\
    	{(a_{-1},b_{-1})} && {(a_{-1},b_{-1})} & {(a_{-1}',b_{-1}')} && {(a_{-1},b_{-1})} & {(a_{-1}',b_{-1}')} \\
    	{(a_0,b_0)} & \Rightarrow & {(a_0,b_0)} & {(a_0',b_0')} & {\text{or}} & {(a_0,b_0)} & {(a_0',b_0')} \\
    	{(a_1,b_1)} && {(a_1,b_1)} &&& {(a_1,b_1)}
    	\arrow[maps to, from=2-3, to=3-3]
    	\arrow[maps to, from=2-7, to=3-7]
    	\arrow[maps to, from=3-1, to=4-1]
    	\arrow[maps to, from=3-3, to=4-3]
    	\arrow[maps to, from=3-4, to=4-4]
    	\arrow[maps to, from=3-6, to=4-6]
    	\arrow[maps to, from=3-7, to=4-7]
    	\arrow[from=4-1, to=5-1]
    	\arrow[maps to, from=4-3, to=5-3]
    	\arrow[maps to, from=4-4, to=5-3]
    	\arrow[maps to, from=4-6, to=5-6]
    	\arrow[maps to, from=4-7, to=5-6]
    \end{tikzcd}
\end{lemma}

\begin{proof}
    (1) follows immediately from Lemma \ref{lem:fish}. For (2), by Corollary \ref{cor:structureofpreim} we know that $(a_{-1}',b_{-1}')$ exists if and only if $\sqrt{(a_0')^2-(b_0')^2}$ exists. But since $a_0'=b_0$ and $b_0'=a_0$, we have
    \begin{equation*}
        \sqrt{(a_0')^2-(b_0')^2}=\sqrt{b_0^2-a_0^2}=i\sqrt{a_0^2-b_0^2}.
    \end{equation*}
    and $\sqrt{a_0^2-b_0^2}$ exists because $(a_{-1},b_{-1})$ exists. 

    Now for (3). By Corollary \ref{cor:structureofpreim} and parts (1) and (2), we have
    \begin{align*}
        a_{-1}=a_0+\sqrt{a_0^2-b_0^2}, &\qquad b_{-1}=a_0-\sqrt{a_0^2-b_0^2}, \\
        a_{-1}'=b_0+\sqrt{b_0^2-a_0^2}, &\qquad b_{-1}'=b_0-\sqrt{b_0^2-a_0^2}.
    \end{align*}
    By Corollary \ref{cor:structureofpreim}, we have that $(a_{-2},b_{-2})$ (resp $(a_{-2}',b_{-2}')$) exist if and only if $\Delta=a_{-1}^2-b_{-1}^2$ (resp. $\Delta'=(a_{-1}')^2-(b_{-1}')^2$) is a square. We compute
    \begin{align*}
        \Delta&=a_{-1}^2-b_{-1}^2=(a_0+\sqrt{a_0^2-b_0^2})^2-(a_0-\sqrt{a_0^2-b_0^2})^2=4a_0\sqrt{a_0^2-b_0^2}, \\
        \Delta'&=(a_{-1}')^2-(b_{-1}')^2=(b_0+\sqrt{b_0^2-a_0^2})^2-(b_0-\sqrt{b_0^2-a_0^2})^2=4b_0\sqrt{b_0^2-a_0^2}.
    \end{align*}
    If we compute $\Delta\Delta'$, we get
    \begin{equation*}
        \Delta\Delta'=\left(4a_0\sqrt{a_0^2-b_0^2}\right)\left(4b_0\sqrt{b_0^2-a_0^2}\right)=16a_0b_0i\sqrt{a_0^2-b_0^2}^2=i(\text{square}),
    \end{equation*}
    where $16a_0b_0\sqrt{a_0^2-b_0^2}^2$ is a square because $a_0b_0$ is a square. Since $q\equiv 5\mod 8$, $i\not \in \FF_q^{\times 2}$, $\Delta\Delta'$ is not a square. But this implies that exactly one of $\Delta,\Delta'$ is a square, which yields (3) by Corollary \ref{cor:structureofpreim}.
\end{proof}

\begin{corollary}\label{cor:mustfish}
    Let $q\equiv 5\mod 8$, and let $(a_{-1},b_{-1})\mapsto(a_{0},b_{0})\mapsto(a_1,b_1)$ be a chain of length two. Then the connected component containing this chain contains a nontrivial strongly connected component. Further, this component is a jellyfish (i.e. it is not strongly connected.)
\end{corollary}

\begin{proof}
    By Lemma \ref{lem:5mod8}, we obtain a chain $(a_{-2}',b_{-2}')\mapsto (a_{-1}',b_{-1}')\mapsto (a_{0}',b_{0}')\mapsto (a_1,b_1)$, where $(a_i',b_i')$ may or may not equal $(a_i,b_i)$. Repeating this inductively on the chain $(a_{-2}',b_{-2}')\mapsto (a_{-1}',b_{-1}')\mapsto (a_{-0}',b_{-0}')$, we get an arbitrarily long chain. Since there are only finitely many points in $V(\mathbb{F}_q)$, this chain must repeat at some point, and therefore contains a cycle, which is strongly connected. 

    The get that this component is a jellyfish, note that by the ``exactly one" part of the statement of \ref{lem:5mod8}(3), there must be a node $(a_{-2}'',b_{-2}'')\in V(\FF_q)$, equal to either $(a_{-2},b_{-2})$ or $(a_{-2}',b_{-2}')$, that has in degree $0$. Since all vertices of a strongly connected graph have in degree at least $1$, this proves that the connected component is not strongly connected.
\end{proof}

\begin{proposition}\label{thm:q5mod8}
    If $q \equiv 5 \pmod 8$, the connected components in the AGM graph $A(\mathbb{F}_q)$ can be classified into three types: isolated points, fish, and jellyfish.
\end{proposition}
\begin{proof}
    We do casework. 
    \begin{enumerate}
    \item If a connected component has no edges, it is an isolated point. 
    \item If the maximum length of any chain in a connected component is exactly 1, then it must contain at least one edge, and so by Lemma \ref{lem:fish} it contains four points in a fish configuration. None of these four points can be connected to any other points, or else there would be a chain of length 2. So the component is of size four and is therefore a fish. 
    
    \item If there is a chain of length 2 in a connected component, it must be a jellyfish by by Corollary \ref{cor:mustfish}.
    \end{enumerate}
    All these cases together complete the proof.
\end{proof}
The $q \equiv 1 \pmod 8$ case is more interesting. We no longer have a result analogous to Lemma \ref{lem:5mod8}, so we can have acyclic components of arbitrary size. 

\begin{definition}
    A \textit{turtle} in $A(\FF_q)$ is a nontrivial connected component that is strongly connected. 
\end{definition}

\begin{figure}[h]
    \centering
    \scalebox{0.8}{
        \begin{tikzcd}[ampersand replacement=\&]
        	\& {(2+i,1+i)} \& {(2+i,2+2i)} \& {(2+2i,2+i)} \& {(2+2i,1+2i)} \\
        	{(i,1)} \& {(i,2)} \& {(2,i)} \& {(2,2i)} \& {(2i,2)} \& {(2i,1)} \\
        	\& {(1+2i,2+2i)} \& {(1+2i,1+i)} \& {(1+i,1+2i)} \& {(1+i,2+i)} \\
        	{(2i,2)} \& {(2i,1)} \& {(1,2i)} \& {(1,i)} \& {(i,1)} \& {(i,2)} \\
        	\& {(2+i,1+i)} \& {(2+i,2+2i)} \& {(2+2i,2+i)} \& {(2+2i,1+2i)}
        	\arrow[from=1-2, to=2-1]
        	\arrow[from=1-2, to=2-2]
        	\arrow[from=1-3, to=2-3]
        	\arrow[from=1-3, to=2-4]
        	\arrow[from=1-4, to=2-3]
        	\arrow[from=1-4, to=2-4]
        	\arrow[from=1-5, to=2-5]
        	\arrow[from=1-5, to=2-6]
        	\arrow[from=2-2, to=3-2]
        	\arrow[from=2-2, to=3-3]
        	\arrow[from=2-3, to=3-2]
        	\arrow[from=2-3, to=3-3]
        	\arrow[from=2-4, to=3-4]
        	\arrow[from=2-4, to=3-5]
        	\arrow[from=2-5, to=3-4]
        	\arrow[from=2-5, to=3-5]
        	\arrow[from=3-2, to=4-1]
        	\arrow[from=3-2, to=4-2]
        	\arrow[from=3-3, to=4-3]
        	\arrow[from=3-3, to=4-4]
        	\arrow[from=3-4, to=4-3]
        	\arrow[from=3-4, to=4-4]
        	\arrow[from=3-5, to=4-5]
        	\arrow[from=3-5, to=4-6]
        	\arrow[from=4-2, to=5-2]
        	\arrow[from=4-2, to=5-3]
        	\arrow[from=4-3, to=5-2]
        	\arrow[from=4-3, to=5-3]
        	\arrow[from=4-4, to=5-4]
        	\arrow[from=4-4, to=5-5]
        	\arrow[from=4-5, to=5-4]
        	\arrow[from=4-5, to=5-5]
        \end{tikzcd}
    }
    \caption{A turtle in $A(\mathbb{F}_{9})$ of size $16$. Here the duplicate points are identified, and we note that the component can be naturally drawn on a torus. }
\end{figure}
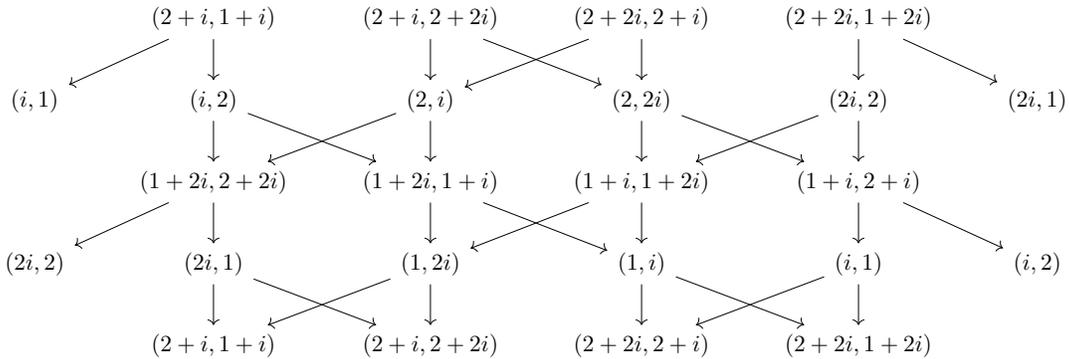
\noindent We will show that the turtles are associated to supersingular curves. 
\begin{definition}
    A \textit{large acyclic component} is an acyclic connected component larger than size 4.
\end{definition}
Figure \ref{fig:504} shows a large acyclic component of size 504, drawn as an acyclic graph sorted into levels. The half-size intermediate level is typical; see Section \ref{sec:conclusion} for more discussion. 
\begin{figure}[ht]
    \centering
    \includegraphics[width=0.9\linewidth]{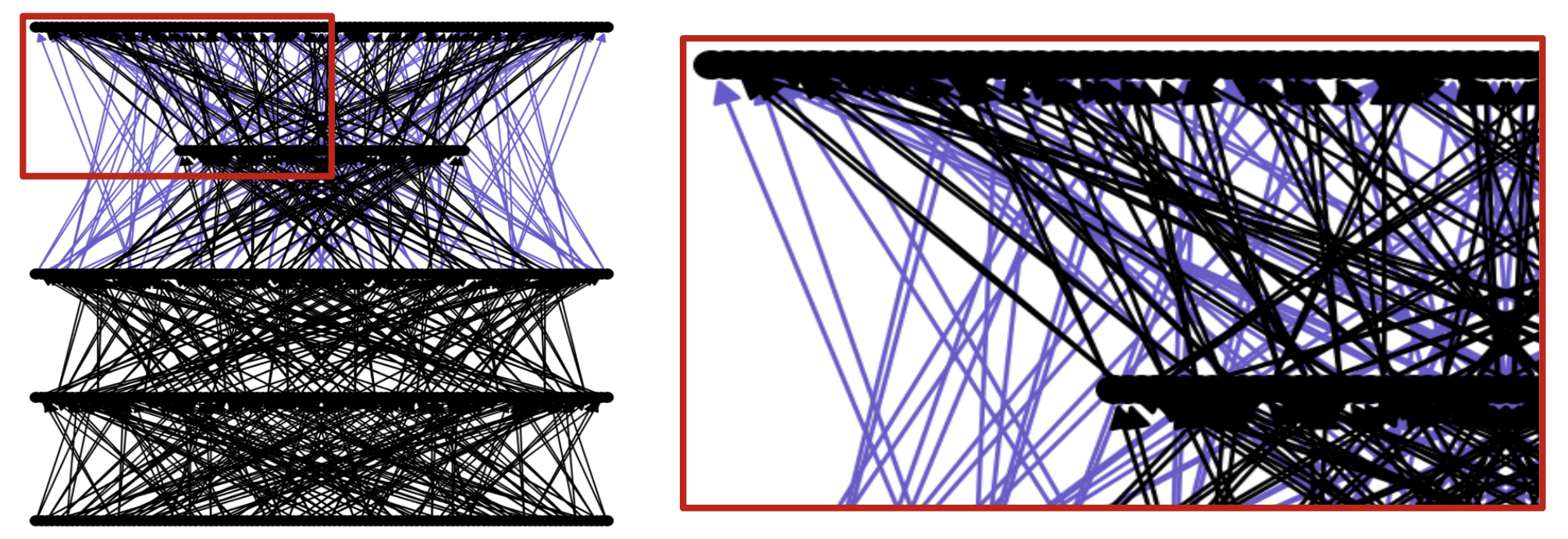}
    \caption{A large acyclic component of size 504, arising in $A(\mathbb{F}_{113})$, and a closeup of the same component. The edges are directed upwards; edges highlighted in blue are going directly from the middle level to the top level. }
    \label{fig:504}
\end{figure}

The jellyfish in the case $q \equiv 1 \pmod 8$ are larger and more complicated than the jellyfish in the other cases, but they have the same essential structure. We will show that the strongly connected components in the jellyfish are simple cycles, for all aquarium graphs $A(\mathbb{F}_q)$. 

\begin{proposition}\label{prop:cycles}
Let $H$ be a strongly connected component in an AGM aquarium $A(\mathbb{F}_q)$ with $2 \nmid q$, such that none of the points correspond to supersingular elliptic curves. Then:
\begin{enumerate} 
\item The elliptic curves corresponding to vertices in $H$ all have complex multiplication by the same order $\mathcal{O}$. 
\item There is a unique ideal $\mathfrak{p}_2\subset \mathcal{O}$ whose action on elliptic curves corresponding to points in the cycle gives the isogeny associated to the edges out of that point. In addition, $\left(\frac{{\rm disc}(K)}{2} \right) = 1$, where $K$ is the imaginary quadratic field containing $\mathcal{O}$.
\item $H$ is a simple cycle, and the action of $\mathfrak{p}_2$ sends elliptic curves associated to points in $H$ to the elliptic curves associated to their successors in $H$ (rotating the head forward by one edge). 

\end{enumerate}
\end{proposition}
\begin{proof}
In the following, we make heavy use of Lemma \ref{lem:cmtheory}.
\begin{enumerate}
\item Assume for the sake of contradiction that there exist points $P = (a, b)$ and $P' = (a', b')$ in $H$ with corresponding elliptic curves $E_\lambda$ and  $E_{\lambda'}$ such that $E_\lambda$ has complex multiplication by $\mathcal{O}$ and $E_{\lambda'}$ has complex multiplication by $\mathcal{O}'$, with $\mathcal{O} \neq \mathcal{O}'$.

Since $P$ and $P'$ are in the same strongly connected component, there must be some simple cycle containing them both, which we will call $C$. Assume without loss of generality that $\mathcal{O}' \subsetneq \mathcal{O}$ and $\mathcal{O}$ is the maximal order among orders in the simple cycle. Then there must be at least one descending isogeny (see Lemma \ref{lem:cmtheory}) in the sequence of isogenies that takes $E_\lambda$ to $E_\lambda'$, and at least one ascending isogeny in the sequence of isogenies that takes $E_\lambda'$ to $E_\lambda$. 

For any order $\mathcal{O}''$ of an elliptic curve in the simple cycle, we know $[\mathcal{O}: \mathcal{O''}] = 2^r$ for some $r>0$, since they are connected by a sequence of 2-isogenies and $\mathcal{O}$ is maximal among orders in the cycle. So we have that $2|[\mathcal{O}_K : \mathcal{O''}] = [\mathcal{O}_K : \mathcal{O}][\mathcal{O} : \mathcal{O}'']$, and so by Lemma \ref{lem:cmtheory}, there cannot be any horizontal 2-isogenies between elliptic curves with CM by the same non $2$-maximal order $\mathcal{O}'$.

\begin{figure}[h!]
\[\begin{tikzcd}
	\bullet & \bullet && \bullet & \bullet & P \\
	&& {P''}
	\arrow[from=1-1, to=1-2]
	\arrow["\varphi", from=1-2, to=2-3]
	\arrow[from=1-4, to=1-5]
	\arrow[dotted, no head, from=1-5, to=1-6]
	\arrow["\psi", from=2-3, to=1-4]
\end{tikzcd}\]
\end{figure}
\vspace{-20pt}
As a result, there must be some pair of ascending and descending isogenies directly adjacent to each other, since there cannot be any horizontal isogenies in between. Let the descending isogeny be $\varphi$ and the ascending isogeny be $\psi$.

Consider the dual of the descending isogeny $\varphi$. It must be an ascending 2-isogeny, but we know that there is a unique ascending isogeny from $P''$ by Lemma \ref{lem:ascending}, which must be $\psi$. So $\psi$ and $\varphi$ are dual; but this is a contradiction, since by Corollary \ref{cor:nomultby2} dual isogenies are never adjacent in the AGM graph. So we must have $\mathcal{O}' = \mathcal{O}$ and all of the elliptic curves in $H$ have complex multiplication by the same order $\mathcal{O}$.
\item Let $P_0 = (a_0, b_0)$, $P_1 = (a_1,b_1)$, $P_2 = (a_2,b_2)$ be three points in $H$, with $P_0 \mapsto P_1 \mapsto P_2$ under the AGM map. 

Then by above argument the elliptic curves $E_{\lambda_0}, E_{\lambda_1}$ and $E_{\lambda_2}$ all have complex multiplication by the same order, $\mathcal{O}$, and the AGM isogenies between them are horizontal. In particular, there exists an invertible $\mathcal{O}$-ideal $\mathfrak{p}_2$ of norm 2 such that the isogeny $\varphi_0: E_{\lambda_0} \rightarrow E_{\lambda_1}$ defined by the AGM map corresponds to the action of $\mathfrak{p}_2$ on $E_{\lambda_0}$, and $E_{\lambda_0}[\mathfrak{p}_2] = \langle (0,0) \rangle$, the kernel of $\varphi_0$. 

Similarly, there is an ideal $\mathfrak{p}_2'$ of norm 2 such that the action of $\mathfrak{p}_2'$ on $E_{\lambda_1}$ gives the isogeny $\varphi_1: E_{\lambda_1} \rightarrow E_{\lambda_2}$. Since there are at most two ideals of norm 2, we either have $\mathfrak{p}_2'= \mathfrak{p}_2$ or $\mathfrak{p}_2'= \overline{\mathfrak{p}_2}$. If $\mathfrak{p}_2'= \overline{\mathfrak{p}_2}$, then their corresponding isogenies $\varphi_0$ and $\varphi_1$ must be dual, a contradiction (by Corollary \ref{cor:nomultby2}). So $\mathfrak{p}_2'= \mathfrak{p}_2$, and every isogeny corresponding to an edge in $H$ must be an action by the same ideal $\mathfrak{p}_2$. 

Further, this argument demonstrates that there must exist two distinct ideals of norm 2 in $\mathcal{O}$, implying $\left(\frac{{\rm disc}(K)}{2} \right) = 1$.  

\item Assume by contradiction that $H$ is not a simple cycle. Then there must exist some point $(a_0,b_0) \in H$ in multiple cycles, such that $(a_0, b_0) \mapsto (a_1 ,b_1), (a_1, -b_1)$ under the AGM map and both $(a_1 ,b_1)$ and $(a_1, -b_1)$ are part of distinct cycles containing $(a_0, b_0)$. 

\[\begin{tikzcd}
	{} & {(a_{-1}, b_{-1})} && {(b_{-1}, a_{-1})} & {} \\
	&& {(a_0, b_0)} \\
	{} & {(a_1, b_1)} && {(a_1, -b_1)} & {}
	\arrow[dashed, tail reversed, no head, from=1-2, to=1-1]
	\arrow[from=1-2, to=2-3]
	\arrow[from=1-4, to=2-3]
	\arrow[dashed, from=1-5, to=1-4]
	\arrow[from=2-3, to=3-2]
	\arrow[from=2-3, to=3-4]
	\arrow[curve={height=-30pt}, dashed, from=3-1, to=1-1]
	\arrow[dashed, tail reversed, no head, from=3-1, to=3-2]
	\arrow[dashed, from=3-4, to=3-5]
	\arrow[curve={height=30pt}, dashed, from=3-5, to=1-5]
\end{tikzcd}\]

Then $\lambda(a_1, b_1) = \lambda(a_1, -b_1)$. Their children define isomorphic elliptic curves, but are sent to different values by the lambda function: $\lambda_2$ and $\frac{\lambda_2}{\lambda_2-1 }$ respectively, by Lemma \ref{lem:cousins}. Then by Lemma \ref{lem:rivalscorresponding2tors}, there is an isomorphism between these curves that sends $(0,0) \in E_{\lambda_2}(\FF_q)$ to $(\frac{\lambda_2}{\lambda_2-1 }, 0) \in E_{\scriptscriptstyle \frac{\lambda_2}{\lambda_2-1 }}(\FF_q)$. So the AGM isogeny with kernel $\langle (0,0) \rangle$ for $E_{\lambda_2}$ is different than the AGM isogeny with kernel $\langle (0,0) \rangle$ for $E_{\scriptscriptstyle \frac{\lambda_2}{\lambda_2-1 }}$, and in particular they cannot both be given by the action of $\mathfrak{p}_2$ on their isomorphism class. Since every isogeny in $H$ is given by the action of $\mathfrak{p}_2$, we conclude that only one of these points can actually be in $H$, and $H$ must be a simple cycle after all. (Here we do not need to consider the cases $j=0, 1728$, because $\left(\frac{{\rm disc}(K)}{2} \right) \neq 1$ in these cases and so no point with a $j$-invariant in $\{0,1728\}$ can be in a jellyfish.)

\end{enumerate}
All the three parts are now proved.
\end{proof}

\begin{corollary}\label{prop:nojfishfate}
Let $E_\lambda$ be an elliptic curve with complex multiplication by an order $\mathcal{O}$ in an imaginary quadratic field $K$, and let $P$ be a point corresponding to $E_\lambda$. If $\left(\frac{\text{disc}(K)}{2} \right) = 0,-1$, then $P$ cannot be in a jellyfish.
\end{corollary}
\begin{proof}
Proposition \ref{prop:cycles} tells us that for any point in a jellyfish head, the associated elliptic curve $E_\lambda'$ must have CM by an order $\O$ in an imaginary quadratic field $K'$ such that $\left(\frac{\text{disc}(K')}{2} \right) = 1$. Then if $\left(\frac{\text{disc}(K)}{2} \right) = 0,-1$, any point connected to $P$ corresponds to an elliptic curve with CM by an order in $K$ where $\left(\frac{\text{disc}(K')}{2} \right) \neq 1$, and so no point connected to $P$ can be part of a jellyfish head. Then $P$ cannot be part of a jellyfish. 
\end{proof}

\begin{figure}[h]
    \centering
    \includegraphics[width=0.5\linewidth]{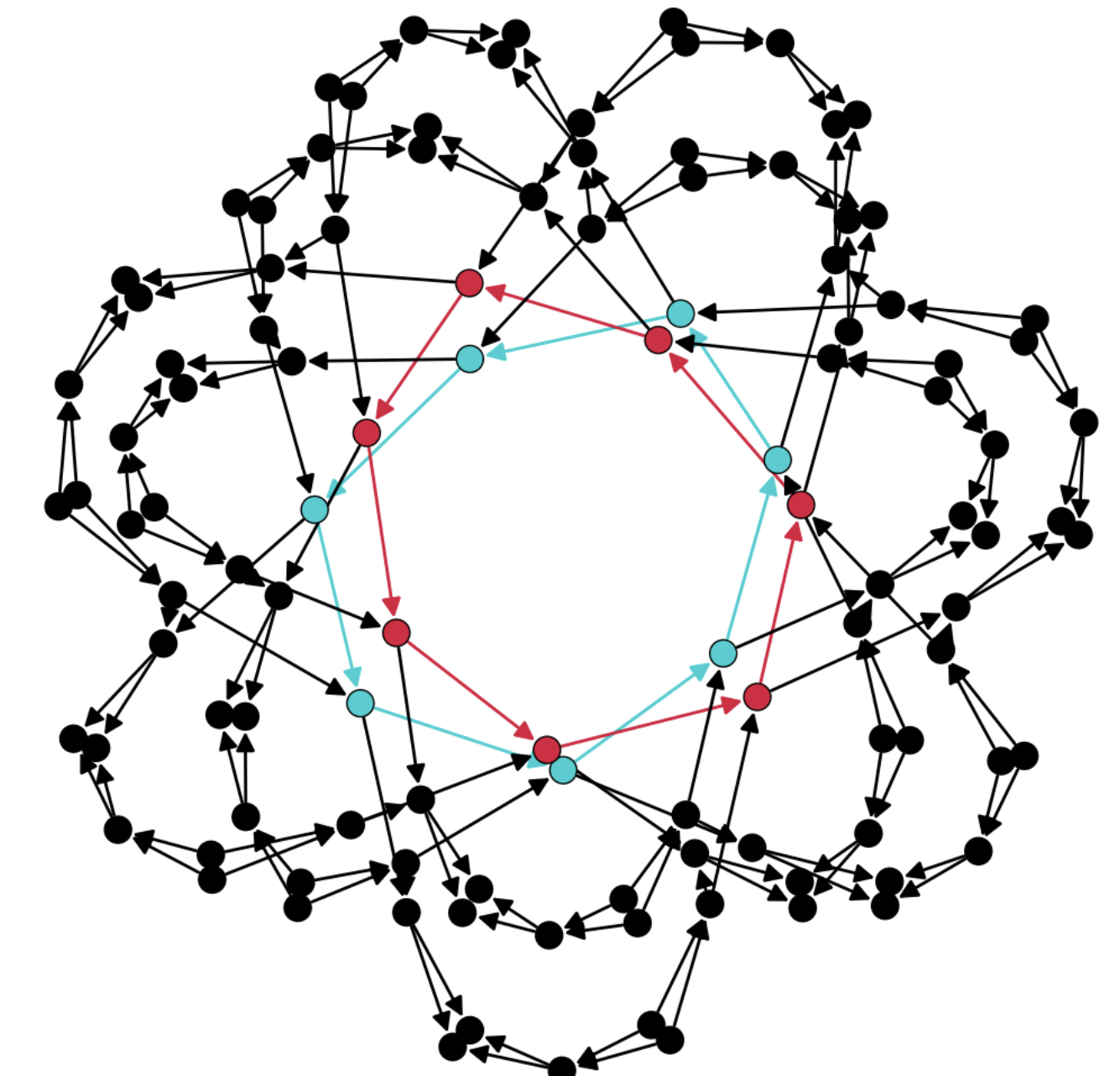}
    \caption{A "tangled" jellyfish of size 112 in $A(\mathbb{F}_{29})$, with two disjoint cycles of size 7. The cycles are highlighted in blue and red respectively. }
    \label{fig:tangledfish}
\end{figure}

\begin{corollary}\label{cor:noturtle}
Let $T$ be a turtle. Then all of the points in $T$ correspond to supersingular elliptic curves.
\end{corollary}
\begin{proof}
Assume otherwise. Then since supersingularity is preserved under isogenies, every point  in $T$ must be ordinary. But turtles are strongly connected components that are not simple cycles, so this is a contradiction of Proposition \ref{prop:cycles} (3). 
\end{proof}

\section{Ideal Classes and Jellyfish}\label{sec:idealclasses}

In this section, we prove results about jellyfish sizes and occurence using the theory of complex multiplication. We first prove conditions on the appearance of turtles, and then prove a result on the sizes and multiplicity of jellyfish. We then derive conditions for a point in $V(\mathbb{F}_q)$ to be in a jellyfish in some sufficiently large aquarium $A(\mathbb{F}_{q^n})$. We also briefly remark on the connected components in the AGM graph for the algebraic closure of a finite field, $A(\overline{\mathbb{F}}_q)$. 

In Proposition \ref{prop:whenturtles}, we characterize the appearance of turtles. In order to prove this proposition, we must first prove a few lemmas. 

\begin{lemma}\label{lem:sschild}
Let $(a,b)$ be a point in $V(\mathbb{F}_{p^2})$ for an odd prime $p$ such that $E_{\lambda(a,b)}$ is supersingular. Then $(a,b)$ has children in $A(\mathbb{F}_{p^2})$.
\end{lemma}
\begin{proof}
We know that $(a,b) = (a, a\sqrt{\lambda})$ for some square root of $\lambda$ over $\mathbb{F}_{p^2}$. By Lemma \ref{lem:4thpowers}, $\lambda$ is a fourth power over $\mathbb{F}_{p^2}$, and in particular $(\frac{a+ a\sqrt{\lambda}}{2}, \pm a\sqrt[4]{\lambda}) \in V(\mathbb{F}_{p^2})$ are the children of $(a,b)$ in $A(\mathbb{F}_{p^2})$. (Both $\sqrt{\lambda}$ and $-\sqrt{\lambda}$ have square roots because $p^2 \equiv 1 \pmod 4$ and so $-1$ is a square over $\mathbb{F}_{p^2}$.) 
\end{proof}
\begin{lemma}\label{lem:turtledegree}
    Let $C$ be a connected component in $A(\FF_q)$ such that every point in $C$ has out-degree 2. Then $C$ is a turtle. 
\end{lemma}

\begin{proof}
    Since every point in $C$ has out-degree 2, and the maximum in-degree of any point is also 2, we know that every vertex has in-degree and out-degree exactly 2 because the sum of in-degrees must equal the sum of out-degrees. Then $C$ is Eulerian and connected \cite[pg 91]{Bondy2008GraphT}, and therefore strongly connected. 
\end{proof}
With these lemmas, we can now prove our result. 
\begin{proposition} \label{prop:whenturtles}
The AGM graph $A(\mathbb{F}_{q})$ has turtles if and only if $q$ is a square. Further, for $q = p^{2r}, p$ prime, every turtle in $A(\mathbb{F}_{q})$ is a scalar multiple of a turtle in $A(\mathbb{F}_{p^2})$, and every point in $V(\mathbb{F}_{q})$ whose associated elliptic curve is supersingular is part of some turtle. (Here "scalar multiples" refers to subgraphs that are sent to each other under the action of $\mathbb{F}_q^\times$ on the graph.)

\end{proposition}

\begin{proof}

We know that all supersingular elliptic curves over $\overline{\mathbb{F}}_p$ can be written in Legendre form with $\lambda \in \mathbb{F}^{\times 4}_{p^2}$ by Lemma \ref{lem:4thpowers}. Then any point $P =(a,b) \in V(\mathbb{F}_q)$ whose associated elliptic curve $E_\lambda$ is supersingular must be a scalar multiple $P = a P'$ of a point $P' = (1, \sqrt{\lambda}) \in V(\mathbb{F}_{p^2})$, with $\sqrt{\lambda} = \frac{b}{a}$ a square root of $\lambda$ over $\mathbb{F}_{p^2}$. 

Now we consider the connected component of $P'$ in $A(\mathbb{F}_{p^2})$. By Lemma \ref{lem:sschild}, $P'$ has children in $A(\mathbb{F}_{p^2})$, and every point in $P'$'s connected component is also associated to a supersingular curve and therefore \textit{also} has children. So $P'$ is in a turtle, by Lemma \ref{lem:turtledegree}, and the connected component containing $P$ must be a scalar multiple of that turtle.

Using Corollary \ref{cor:noturtle}, we conclude that the points in turtles for any $A(\mathbb{F}_q)$ with $q$ square are exactly the points in $V(\mathbb{F}_q)$ corresponding to supersingular curves. Since there is always a supersingular curve over $\overline{\mathbb{F}}_p$ for all $p$, there is always at least one supersingular point in $V(\mathbb{F}_{p^2})$, and therefore at least one turtle in $A(\mathbb{F}_q)$ for $q$ square. 

For $q$ \textit{not} square, we have two cases. If $q$ is 3 mod 4, $-1$ is not a square, and so every point has at least one child that is terminal and there are no turtles. 

If $q$ is 1 mod 4 and not a square, then $p$ is 1 mod 4, and we have that there is no supersingular elliptic curve $E_{\lambda}$ with $\lambda \in \mathbb{F}_p$. To see this, note that if $E$ can be written in Legendre normal form over $\FF_p$, then $E[2]\subseteq E(\FF_q)$, so $4\mid \#E(\FF_p)$, but by (\cite{Silverman1986TheAO} Exercise 5.15), $\#E(\FF_p)=p+1\equiv 2\mod 4$, a contradiction. So no points in $V(\mathbb{F}_q)$ are associated to supersingular curves and there cannot be any turtles in  $A(\mathbb{F}_q)$ by Corollary \ref{cor:noturtle}.
\end{proof}

The following lemma is required to establish Theorem \ref{jellyfishmultiplicity}. 

\begin{lemma}\label{lem:norivals}
Let $P$ and $P'$ be two points in $V(\mathbb{F}_q)$ with $\lambda = \lambda(P), \lambda' = \lambda(P')$. If $E_{\lambda}$ and $E_{\lambda'}$ are isomorphic ordinary elliptic curves but $P'$ is not a scalar multiple of $P$, then they cannot be in the same cycle. 
\end{lemma}
\begin{proof}
We do casework. 
\begin{enumerate}
\item $\lambda' \neq \lambda, \frac{1}{\lambda}$.

If $\lambda' \neq \lambda, \frac{1}{\lambda}$, we know that the point $ (0,0)  \in E_{\lambda_1}(\FF_q)$ maps to a point in $\{(1,0), (\lambda', 0)\} \in E_{\lambda'}(\FF_q)$ under the canonical isomorphism, by Lemma \ref{lem:rivalscorresponding2tors}, so the AGM isogenies for these two elliptic curves must be different. By Proposition \ref{prop:cycles}, the isogenies in a cycle are given by the action of a specific ideal on an isomorphism class, which gives a unique isogeny from that isomorphism class. Since the AGM isogenies for $E_{\lambda}$ and $E_{\lambda'}$ are different, only one of them can correspond to the action of the ideal, and therefore they cannot both be in the same cycle. 
\item $\lambda = \lambda'$ but $P'$ is not a scalar multiple of $P = (a,b)$.

In this case, we must have $P' = (\gamma a, -\gamma b)$ for some scalar $\gamma$. But this is not possible, because the action of $\gamma \in \mathbb{F}_q^\times$ on points commutes with the AGM map and so $(a,-b)$ would also be in a cycle. This would imply that both $(a,b)$ and $(a,-b)$ are in cycles, which would mean their shared parent is in multiple cycles, contradicting Proposition \ref{prop:cycles} (3). 

\item $\lambda' =\frac{1}{\lambda}$.

This case requires a little more work. Assume by contradiction that $P = (a,b)$ and $P'$ are both in the cycle, and let the successor of $P$ in the cycle be $P_1 = (a_1,b_1)$. Now we have cases.  
\begin{enumerate}
\item \textit{$P' = (\gamma b, \gamma a)$ for some constant $\gamma$, and its successor in the cycle is $P_1' = (\gamma a_1, \gamma b_1)$.}

In this case $P_1' = \gamma P_1$, so we should have  $P' = \gamma P$ (because the AGM map commutes with the multiplication action of $\mathbb{F}_q^{\times}$ on the graph), which is clearly not true. So $P'$ and $P$ are not in the same cycle.

\item \textit{$P' = (\gamma b, \gamma a)$ for some constant $\gamma$, and its successor in the cycle is $P_1' = (\gamma a_1, -\gamma b_1)$.}

By case (2) above $P_1$ and $P_1'$ cannot be in the same cycle, so this is not possible. 

\item \textit{$P' = (\gamma b, -\gamma a)$ for some constant $\gamma$.}

Let $\lambda_1 = \lambda(a_1, b_1)$. Then both of the children $(a_1', b_1'),(a_1', -b_1') $ of $P'$ satisfy $\lambda(a_1', b_1') = \lambda(a_1', -b_1') = \frac{\lambda_1}{\lambda_1-1}$, by Lemma \ref{lem:cousins}. But no point with lambda equal to $\frac{\lambda_1}{\lambda_1-1}$ can be in the same cycle as $P_1$, by case (1) above. 
\end{enumerate}
\end{enumerate}
The proof is complete.
\end{proof}

\begin{reptheorem}{jellyfishmultiplicity} (Generalization of \cite{hypergeometryagm}, Theorem 1.3)
\textit{Let $q = p^r$ be a prime power, and let $P_0$ be a point on a jellyfish head $H$, associated to an elliptic curve $E_\lambda$. Let $|H|$ be the number of vertices in $H$, and let $M_H$ be the number of distinct scalar multiples of the jellyfish head $H$ in the AGM graph. Then if $\mathcal{O} := \text{End}(E_\lambda)$ is an order in an imaginary quadratic field and $h_2(\mathcal{O})$ denotes the order of $[\mathfrak{p}_2]$ in $cl(\mathcal{O})$, where $\mathfrak{p}_2$ is a prime above $(2)$ in $\mathcal{O}$, we have }

\begin{equation} \label{eq:divresult}
h_2(\mathcal{O}) \mid |H|,
\end{equation}
\begin{equation}\label{eq:countresult}
M_H \cdot |H| = (q-1) \cdot h_2(\mathcal{O}).
\end{equation}
\end{reptheorem}
\begin{proof}
Let $j_0$ be the $j$-invariant of $E_\lambda$.  We know by Proposition \ref{prop:cycles} that there is a prime ideal $\mathfrak{p}_2$ such that the action of $\mathfrak{p}_2$ on the elliptic curves corresponding to the points in $H$ gives the isogenies associated to the edges out of those points. Since $[\mathfrak{p}_2]$ defines a free action on the $j$-invariants of the elliptic curves associated to $H$, we see that we have to travel exactly $h_2(\mathcal{O}) = \text{ord}([\mathfrak{p}_2])$ edges to return to a point $P_1$ whose elliptic curve has the same $j$-invariant $j_0$. 

Since $P_0$ and $P_1$ are in the same connected component, their elliptic curves $E_{\lambda_0}$ and $E_{\lambda_1}$ have the same trace of Frobenius, and since they also have the same $j$-invariant, we conclude that they are isomorphic (see e.g. \cite[pg 542]{waterhouse1969abelian} for a proof of this fact for ordinary elliptic curves over a finite field). 

By Lemma \ref{lem:norivals}, we see that $P_1$ is a scalar multiple of $P_0$. Let this scalar be $\gamma$, and let $n$ be the (necessarily finite) order of $\gamma$ in $\mathbb{F}_{q}^\times$. Then the size of the jellyfish is $n\cdot h_2(\mathcal{O})$, and the multiplicity of the jellyfish is $ \frac{q-1}{n} \cdot h_2(\mathcal{O})$, giving us equations \eqref{eq:divresult} and \eqref{eq:countresult}.
\end{proof}

We will need the following lemma for the proof of Proposition \ref{prop:jfate}. 

\begin{lemma}\label{lem:eventualhead}
Let $P$ be a point in $V(\mathbb{F}_q)$ whose elliptic curve $E_\lambda$ has complex multiplication by an imaginary quadratic order $\mathcal{O}$ with odd index in the maximal order $\mathcal{O}_K$, and suppose $\left( \frac{\text{disc}(K)}{2}\right) = 1$. Let $R_P$ be a set of representatives for $\Lambda_{\lambda(P)}$ as described in Lemma \ref{lem:rivalsconnected}. Then there is some $n$ such that exactly four of the points in $R_P$ are part of a jellyfish head in $V(\mathbb{F}_{q^m})$ for all $n\mid m$. 
\end{lemma}
\begin{proof}
Let $\mathfrak{p}_2$ and $\overline{\mathfrak{p}}_2$ be the two ideals of norm 2 in $\mathcal{O}$. Then there is exactly one reciprocal pair $\lambda', \frac{1}{\lambda '}$ in the set $\{\lambda, \frac{1}{\lambda}, 1-\lambda, \frac{1}{1-\lambda}, \frac{\lambda}{\lambda -1 }, \frac{\lambda - 1 }{\lambda }\}$ such that the AGM isogeny leading out of $\lambda', \frac{1}{\lambda '}$ corresponds to the action of the ideal $\mathfrak{p}_2$ by \ref{lem:rivalscorresponding2tors}. 

Let $(a_1, b_1),(a_1, -b_1)$  be two of the preimages of $E_{\lambda}$ in $R_P$, and let  $(b_1, a_1),(-b_1, a_1)$ be two of the preimages of $E_{\frac{1}{\lambda}}$ in $R_P$. Then we have the following substructure in the graph of $V(\mathbb{F}_{q^4})$. 

\[\begin{tikzcd}
	& {(a_0, b_0)} && {(b_0, a_0)} \\
	{(-b_1, a_1)} & {(a_1,-b_1)} && {(a_1, b_1)} & {(b_1, a_1)} \\
	{(a_2', -b_2')} & {(a_2', b_2')} && {(a_2, b_2)} & {(a_2, -b_2)}
	\arrow[from=1-2, to=2-2]
	\arrow[from=1-2, to=2-4]
	\arrow[from=1-4, to=2-2]
	\arrow[from=1-4, to=2-4]
	\arrow[from=2-1, to=3-1]
	\arrow[from=2-1, to=3-2]
	\arrow[from=2-2, to=3-1]
	\arrow[from=2-2, to=3-2]
	\arrow[from=2-4, to=3-4]
	\arrow[from=2-4, to=3-5]
	\arrow[from=2-5, to=3-4]
	\arrow[from=2-5, to=3-5]
\end{tikzcd}\]

All of the edges in this graph correspond to horizontal isogenies given by $\mathfrak{p}_2$, and the two sets of children in the bottom row satisfy $\lambda(a_2', b_2') = \frac{\lambda(a_2, b_2)}{\lambda(a_2, b_2)-1}$, by Lemma \ref{lem:cousins}. Then their corresponding elliptic curves are isomorphic but define different isogenies, so at least one pair has a horizontal AGM isogeny. This isogeny cannot correspond to the action of $\overline{\mathfrak{p}}_2$, so it corresponds to ${\mathfrak{p}}_2$ (and the other isogeny is non-horizontal).

Assume without loss of generality that $(a_2,b_2)$ and $(a_2, -b_2)$ are the pair with the horizontal AGM isogeny. Then we can draw a similar diagram again and conclude that exactly one of them has a set of children whose AGM isogeny is also horizontal (and corresponds to ${\mathfrak{p}}_2$). 

Repeating this process $k$ times, where $k$ is the order of $[\mathfrak{p}_2]$ in $\cl(\mathcal{O})$, we get a chain of horizontal isogenies that returns to a pair of children with the same $j$-invariant as the original pair $(a_1,b_1)$, which we can call $(a_k, b_k)$ and $(a_k', b_k')$. 

Then we again get a substructure that looks like our original diagram. One of the children must have another child $(a_{k+1}, b_{k+1})$ whose outgoing isogeny is defined by $\mathfrak{p}_2$, and the original child satisfying this condition must be a scalar multiple of one of the original four points. Then we have a chain of isogenies that leads from a point to a scalar multiple of itself, which means we have a jellyfish head.

Since the head of the jellyfish has finite length, there is an $n$ such that the entire chain of edges and therefore the jellyfish head exists in $A(\mathbb{F}_{q^n})$. So this jellyfish head exists and contains a point $P'$ in $R_P$ for some sufficiently large $A(\mathbb{F}_{q^n})$.  

Then the point $-P$ must also be part of a jellyfish head in $A(\mathbb{F}_{q^n})$ (possibly the same one), since it's a scalar multiple of $P$. 

Similarly, the jellyfish head corresponding to the action of $\overline{\mathfrak{p}}_2$ exists and contains a point in $R_P$ for some sufficiently large AGM graph $A(\mathbb{F}_{q^m})$. Taking the LCM of $m$ and $n$, we have a graph where exactly four of the points in $R_P$ are part of jellyfish heads, as desired. 
\end{proof}
We can now prove a result that gives necessary and sufficient conditions for a point to be part of a jellyfish in some field. 

\begin{proposition} \label{prop:jfate}
Let $P_0$ be a point in $V(\mathbb{F}_q)$ and let $E_\lambda$ be the (ordinary) elliptic curve corresponding to it, with complex multiplication by an order $\mathcal{O}$ in an imaginary quadratic field $K$. If $\left(\frac{{\rm disc}(K)}{2} \right) = 1$, then there is some $n$ such that $P_0$ is part of a jellyfish in $A(\mathbb{F}_{q^n})$. 
\end{proposition}
\begin{proof}
Let the index of $\mathcal{O}$ in $\mathcal{O}_K$ be $2^{\ell} m$ for some odd number $m$. Then there is an order $\mathcal{O}'$ such that $[\mathcal{O}': \mathcal{O}] = 2^\ell$, and $[\mathcal{O}_K: \mathcal{O'}] = m$. By \cite[Lemma 6]{isogenyvolcanoes}, there is a unique ascending isogeny from $E_\lambda$ to an elliptic curve with complex multiplication by $\mathcal{O}'$, and it decomposes as a sequence of (unique) ascending 2-isogenies.

By Lemma \ref{lem:rivalsconnected}, there is a set $R(P_0)$ such that $P_0$ is connected to every point in $R(P_0)$ over $A(\mathbb{F}_{q^4})$ and 
for any $\lambda' \in \Lambda_{\lambda(P)}$, there is a point $P_0' \in R(P_0)$ such that $\lambda(P_0) = \lambda'$. By Lemma \ref{lem:rivalscorresponding2tors}, there is some $\lambda \in \Lambda_{\lambda(P_0)}$ whose AGM isogeny corresponds to any given 2-isogeny from the isomorphism class of $E_\lambda$. So for some sufficiently large $n_1$,  $P_0'$ is connected to $R(P_0)$ in $A(\mathbb{F}_{q^{n_1}})$, and therefore to a point $P_1$ whose associated elliptic curve $E_{\lambda(P_1)}$ is the target of the unique ascending isogeny from $E_\lambda$. Similarly, there is some sufficiently large $n_2 \geq n_0$ such that $P_1$ is connected to another point $P_2$ whose elliptic curve $E_{\lambda(P_2)}$ is the target of the unique ascending isogeny from $E_{\lambda(P_1)}$ in $A(\mathbb{F}_{q^{n_2}})$.

\[\begin{tikzcd}
	&&&& {P_2} \\
	&& {P_1} & {P_1'} \\
	{P_0} & {P_0'}
	\arrow[dashed, no head, from=2-3, to=2-4]
	\arrow[from=2-4, to=1-5]
	\arrow[dashed, no head, from=3-1, to=3-2]
	\arrow[from=3-2, to=2-3]
\end{tikzcd}\]

Proceeding in this way, we find that there is a finite $A(\mathbb{F}_{q^{n_\ell}})$ where $P$ is connected to a point $P_\ell$ with complex multiplication by $\mathcal{O}'$. 

Then by Lemma \ref{lem:eventualhead}, there is an $m$ sufficiently large such that $P_\ell$ is connected to a jellyfish head, and so $P$ is part of a jellyfish in $A(\FF_q^{n})$ for some sufficiently large $n$. 
\end{proof}

Now we have all of the results required to prove Theorem \ref{fishfate}. 

\begin{reptheorem}{fishfate}
\textit{
    Let $P$ be a point in $V(\mathbb{F}_q)$ for some odd prime power $q$. Then:
    \begin{enumerate}
        \item If $E_{\lambda(P)}$ is supersingular, then $P$ is part of a turtle in $A(\mathbb{F}_{q^m})$, for any $q^m$ square. 
        \item If $E_{\lambda(P)}$ is ordinary and has complex multiplication by an order $\O$ with fraction field $K$:
        \begin{enumerate}
            \item If $\left(\frac{{\rm disc}(K)}{2} \right) = 1$, there is some $n$ such that the connected component containing $P$ is a jellyfish for all $A(\mathbb{F}_{q^m})$ with $n|m$. 
            \item If $\left(\frac{{\rm disc}(K)}{2} \right) \neq 1$, then the connected component containing $P$ is acyclic in any aquarium $A(\mathbb{F}_{q^n})$. (Note that by our definitions, acyclic simply means that it is neither a jellyfish or turtle.)
        \end{enumerate}
    \end{enumerate}}
\end{reptheorem}
\begin{proof}
\noindent
We  prove the statements one by one.
\begin{enumerate}
    \item Corollary of Proposition \ref{prop:whenturtles}.
    \item \begin{enumerate}
        \item Corollary of Proposition \ref{prop:jfate}. 
        \item Corollary of Corollary \ref{prop:nojfishfate}. 
    \end{enumerate}
\end{enumerate}
Now, the proof is complete.
\end{proof}

This theorem allows us to determine the structure of connected components over the infinite graph $A(\overline{\mathbb{F}}_p)$, because any finite subgraph of $A(\overline{\mathbb{F}}_p)$ is contained in some smaller aquarium $A(\mathbb{F}_q)$, where $q = p^r$. 

So our results from above apply, and if a point $P$ is associated to a supersingular elliptic curve, then its connected component in $A(\overline{\mathbb{F}}_p)$ is a finite turtle. If a point $P$ is associated to an ordinary elliptic curve with complex multiplication by an order $\O \subseteq \O_K$ where $(2)$ does not split, its connected component in $A(\overline{\mathbb{F}}_p)$ is infinite and acyclic. If a point $P$ is associated to an ordinary elliptic curve with complex multiplication by an order $\O \subseteq \O_K$ where $(2)$ \textit{does} split in $\mathcal{O}_K$, then its connected component is infinite and has at least one cycle. 

We also note that for any two orders $\O$ and $\O'$ associated to points in the same component, one of them is contained in the other (without loss of generality assume $\O' \subseteq \O$) and $[\mathcal{O}: \O'] = 2^r$, for some nonnegative $r$, because the elliptic curves are linked by a sequence of 2-isogenies. 

Now consider two points $P_1$, $P_2$ whose associated elliptic curves are ordinary, and let them have complex multiplication by $\O_1$ and $\O_2$. If $\O_1$ and $\O_2$ have different fields of fractions, they cannot be in the same component. Further, if $\O_1$ and $\O_2$ have the same field of fractions $K$, but $\dfrac{[\O_K:\O_1]}{[\O_K:\O_2]} \neq 2^\ell$ for any integer $\ell$, then $P_1$ and $P_2$ cannot be in the same connected component. So the connected components are further stratified by the odd part of the index of the orders in their field of fractions. 

\section{Counting Components}\label{sec:counting}
In this section, we prove a formula for Hurwitz-Kronecker class numbers similar to the one proven in \cite[Thm. 1.1]{hypergeometryagm}, and generalize a result of \cite[Thm. 5]{GOSTJellyfish} on the number of connected components in the AGM graph. These both rely on a result from \cite[p.~197]{ShoofCurveCounting}. 

\begin{definition}[\cite{ShoofCurveCounting} prop. 2.4]
    The \textit{Hurwitz-Kronecker class number} of an imaginary quadratic order $\O$ is 
    \[H(\O) = \sum \limits_{\O \subseteq \O' \subseteq \O_K} \frac{2}{|(\O')^\times|}h(\O'),\]
where $h(\O')$ is the ordinary class number, and the sum is over orders between $\mathcal{O}$ and $\mathcal{O}_K$. We also write $H(-\disc(\O)) = H(\O)$ for discriminants (note the negative sign).
\end{definition}
\begin{definition}
    For a finite field $\mathbb{F}_q$, $n \in \mathbb{Z}$, and $t\in \mathbb{Z}$ let 

    \begin{align}
    &N_{n\times n}(t) = \#\{\mathbb{F}_q\text{-isomorphism classes of elliptic curves $E$ with }\\&\hspace{90pt}\text{Frobenius trace $t$ and }E(\mathbb{F}_q)[n] \cong \mathbb{Z}/n\mathbb{Z} \oplus \Z/n\Z\}. \notag
    \end{align}
\end{definition}

Now we can state the result from \cite{ShoofCurveCounting}.
\begin{lemma}[\cite{ShoofCurveCounting}, Theorem 4.9(i)]\label{lem:schoof}
Let $n \in \Z$, $q = p^r$, with $t \in \Z$ satisfying $t^2 \leq 4q$. Then for $t$ such that $p \nmid t$, $q \equiv 1 \pmod n$, and $t = q+1 \pmod{n^2}$, 

\[N_{n\times n}(t) = H\left(\frac{t^2-4q}{n^2}\right).\]
\end{lemma}

\begin{remark}
    In \cite{ShoofCurveCounting}, the author proves this result for the case of $n$ odd, but their proof works for all integers $n>0$.
\end{remark}

We now will prove two preliminary counting results, Proposition \ref{prop:hurwitzclassnum} and Lemma \ref{lem:hurwitzclassnum4tors} that will be useful for the remainder of this chapter.

\begin{definition}
    Let $q$ be a prime power and $s\in[-2\sqrt q,2\sqrt q]\cap \ZZ$. Define $M_{\FF_q}(s)$ to be the number of distinct $j$-invariants of elliptic curves of the form $E_{\alpha^2}$ with trace of Frobenius given by $s$.
\end{definition}

\begin{proposition}\label{prop:hurwitzclassnum}
    Suppose $\FF_q$ is a finite field with $p\geq 5$. If $-2\sqrt{q}\leq s\leq 2\sqrt q$ is an integer with $s\equiv q+1\mod 8$, then
    \begin{enumerate}
        \item If $E/\FF_q$ has trace of frobenius $s$ and $s\equiv q+1\mod 8$, then there exists $\lambda\in \FF_q^{\times 2}$ with $E\cong E_\lambda$.
        \item $M_{\FF_q}(s)=H\left(\frac{4q-s^2}{4}\right).$
    \end{enumerate}
\end{proposition}

\begin{proof}
    The following proof is a direct generalization of the argument in the $q\equiv 3\mod 4$ case given in \cite[Th. 6]{GOSTJellyfish}.

    \begin{enumerate}
        \item This is \cite[Prop. 3.3]{ASOmodularity}, but they only consider the case of $q=p^1$. Their proof applies identically in our more general case.
        \item By (1), it suffices to count the number of elliptic curves $E/\FF_q$ with trace of Frobenius $s$ and $\Z/2\Z \oplus \Z/2\Z \subseteq E(\mathbb{F}_q)$. By Lemma \ref{lem:schoof}, this is exactly $H((4q-s^2)/4)$, and we are done.
    \end{enumerate}
    
\end{proof}

\begin{definition}
    Let $q\equiv 1 \mod 4$ be a prime power and $s\in [-2\sqrt q,2\sqrt q]$ be a trace of Frobenius. Let $N_{\FF_q}(s)$ be the number of distinct $j$-invariants of elliptic curves of the form $E_{\alpha^2}$ with $1-\alpha^2\in \FF_q^{\times 2}$ and trace of Frobenius $s$.
\end{definition}

\begin{lemma}\label{lem:hurwitzclassnum4tors}
    Suppose $\FF_q$ is a finite field with $q\equiv 1 \mod 4$. If $-2\sqrt q\leq s \leq 2\sqrt q$ is an integer with $s\equiv q+1\mod 16$, then
    \begin{equation*}
        N_{\FF_q}(s)=H\left(\frac{4q-s^2}{16}\right).
    \end{equation*}
\end{lemma}

\begin{proof}
    To prove this, we will show a $1$-to-$1$ correspondence between elliptic curves $E/\FF_q$ with $s\equiv q+1\pmod {16}$ up to isomorphism and elliptic curves $E_{\alpha^2}/\FF_q$ up to isomorphism with the same property. Suppose $E/\FF_q$ is an elliptic curve with trace of Frobenius $s\equiv q+1\mod 16$ and $\ZZ/4\ZZ\oplus \ZZ/4\ZZ\subseteq E(\FF_q)$. Then by Proposition \ref{prop:hurwitzclassnum}(1), we can write $E=E_{\alpha^2}:y^2=x(x-1)(x-\alpha^2)$. Since $E[4]\subseteq E(\FF_q)$, we see that $(1,0)\in 2E(\FF_q)$, which by $2$-descent implies $1-\alpha^2\in \FF_q^{\times 2}$.

    If you have an elliptic curve $E_{\alpha^2}$ with trace of frobenius $s\equiv q+1\pmod {16}$, then you automatically have an isomorphism class of elliptic curves $E/\FF_q$ with those properties by setting $E=E_{\alpha^2}$.

    With this in mind, we see that it suffices to count $j$-invariants of elliptic curves with $\ZZ/4\ZZ\oplus \ZZ/4\ZZ\subseteq E(\FF_q)$, which is exactly $N_{4 \times 4}$ from Lemma \ref{lem:schoof}. Thus
    \begin{equation*}
        N_{\FF_q}(s)=N_{4 \times 4}(s) = H\left(\frac{4q-s^2}{16}\right).
    \end{equation*}
\end{proof}

\begin{proposition}[Generalization of \cite{hypergeometryagm}, Theorem 1.1]
Let $H(D)$ be the Hurwitz-Kronecker class number for the discriminant $D$, and let $q \equiv 1 \pmod 4$. Then we have the formula 

\[\frac{(q-1)(q-5)}{2} = \sum_{\substack{|t| \leq 2\sqrt{q} \\ t \equiv q + 1 \pmod{16}}} 12(q-1)H\left(\frac{4q-t^2}{16}\right).\]
    
\end{proposition}
\begin{proof}
    In fact, we will show that  both sides of this equation count the number of points $P$ in $V(\mathbb{F}_q)$ such that $P$ has a parent in $V(\mathbb{F}_q)$. We know a point $P = (a,b)$ has a parent if and only if $a^2-b^2$ is a square; see Corollary \ref{cor:structureofpreim}. The number of nonzero solutions $(a,b,x)$ to $a^2-b^2=x^2$ up to scalar multiples is $q+1$, so the number of triples $(a,b,x)$ is $(q-1)(q+1)$. We then subtract $6(q-1)$ to account for solutions $(t,0,\pm t),(0,t,\pm it),(t,\pm t,0)$ that are not in the domain $V(\FF_q)$. We then divide by $2$ to account for the sign on $x$ we get $\frac{(q-1)(q+1)}{2}-(3q-3) = \frac{(q-1)(q-5)}{2}$ points with parents.

    For the other side, we notice that by Lemma \ref{rivalsexistence}, a point $P$ exists and has a parent in $\mathbb{F}_q$ if and only if $1-\lambda(P)$ is a square in $\mathbb{F}_q$. By Lemma \ref{lem:hurwitzclassnum4tors}, the number of isomorphism classes of elliptic curves with trace $s$ whose corresponding points have this property is counted by $N_{\mathbb{F}_q}(s) = H\left(\frac{4q-s^2}{16}\right)$. We also note that there are exactly $12(q-1)$ points corresponding to each elliptic curve isomorphism class by results of Section \ref{sec:prelims}.
    
    Since the sum on the right ranges over all traces and counts the number of points with parents whose associated elliptic curve has that trace, it also exactly counts the number of points with parents in $V(\mathbb{F}_q)$, and the two sides are equal.
\end{proof}

Now we turn our attention to proving the second main result of this section: theorem \ref{countingcomponents}, which generalizes the following result of \cite{GOSTJellyfish}.

\begin{proposition}[\cite{GOSTJellyfish} Theorem 5]\label{prop:countingjellyfish3mod4}
    Let $q\equiv 3\mod 4$ and $d(\FF_q)$ be the number of jellyfish over $\FF_q$. Then for all $\varepsilon>0$ and sufficiently large $q$
    \begin{equation*}
        d(\FF_q)>\left(\frac12-\varepsilon\right)\sqrt q.
    \end{equation*}
\end{proposition}

Because all nontrivial connected components are jellyfish when $q\equiv 3\mod 4$, it is not clear whether the correct generalization of Proposition \ref{prop:countingjellyfish3mod4} counts nontrivial connected components or jellyfish. We have chosen to count jellyfish, as our phrasing of the previous theorem suggests.

\begin{definition}
    Let $q=p^n$ be an odd prime power. We define $d(\FF_q)$ to be the number of jellyfish in the AGM aquarium $A(\FF_q)$.
\end{definition}

For the case of $q\equiv 3\mod 4$ this agrees with the $d(\FF_q)$ used in \cite{GOSTJellyfish}. For $q\equiv 5\mod 8$, this definition counts exactly the number of nontrivial non-fish connected components by Proposition \ref{thm:q5mod8}.

\begin{remark}
    Our choice to count jellyfish is motivated by attempting to find a class of objects that grows similarly to how nontrivial connected components grow in the $q\equiv 3\mod 4$ case. Indeed, in the $q\equiv 5\mod 8$ case, there are frequently occurring fish that are characterized by solutions $(a,b)$ to the constraints
    \begin{equation*}
        \begin{cases}
            a^2-b^2 & \text{is not square} \\
            ab & \text{is square} \\
            \frac{(a+b)\sqrt{ab}}{2} & \text{is not square}
        \end{cases}
    \end{equation*}
    which, heuristically, should have a number of solutions quadratic in $q$, and indeed such heuristics can be verified with the Weil conjectures and the principal of inclusion-exclusion. Even worse, in the $q\equiv 1\mod 8$ case, there are many more types of ``frequently occuring small acyclic connected components." In particular, there exist components of size $24$ and $96$ that appear to grow quadratically. Because of this, we have chosen to restrict our attention to bounding the number of Jellyfish, where numerical evidence suggests that their quantity does not grow as quickly as smaller components.
\end{remark}

\begin{reptheorem}{countingcomponents}
    Let $q=p^n$ and let $\varepsilon>0$.
    \begin{enumerate}
        \item If $q\equiv 5\mod 8$, then for all sufficiently large $q$
        \begin{equation*}
            d(\FF_q)\geq \left(\frac{p-1}{8p}-\varepsilon\right)\sqrt{q}.
        \end{equation*}
        \item If $q\equiv 1\mod 8$, then for all sufficiently large $q$
        \begin{equation*}
            d(\FF_q)\geq \left(\frac{p-1}{32p}-\varepsilon\right)\sqrt{q}.
        \end{equation*}
    \end{enumerate}
\end{reptheorem}

There are several steps to this proof. We need a generalization of \cite[Th. 6]{GOSTJellyfish} which itself is built off a result of \cite{ShoofCurveCounting}. This was accomplished in Proposition \ref{prop:hurwitzclassnum}. After this, we need some additional lemmas to classify which elliptic curves actually appear in jellyfish. For the $q\equiv 5\mod 8$ case, these are Lemmas \ref{lem:planktondisjoint} and \ref{lem:fishdisjoint}. For the $q\equiv 1\mod 8$ case, this is Lemma \ref{lem:inhead}.

\begin{lemma}\label{lem:planktondisjoint}
    Let $q\equiv 1\mod 4$, $(a,b)\in V(\FF_q)$, and $\lambda=b^2/a^2$. Then
    \begin{enumerate}
        \item $(a,b)$ has a parent if and only if $E_\lambda[4]\subseteq E_\lambda(\FF_q)$.
        \item If $(a,b)$ does not have a parent, then $(0,0)\in 4E_\lambda(\FF_q)$ if and only if $(a,b)$ has a child.
    \end{enumerate}
\end{lemma}

\begin{proof}
    (1) By Lemma \ref{lem:2torsingeneral}, we have that $\nu_2(|E_\lambda(\FF_q)|)\geq 3$ because $(0,0)\in 2E_\lambda(\FF_q)$. Therefore in the first case it suffices to determine when we also have $(\lambda,0)\in 2E_\lambda(\FF_q)$. 

    We know that $a^2-b^2$ is a square if and only if $(a,b)$ has a parent by Corollary \ref{cor:structureofpreim}, so further $(a^2-b^2)/b^2=\lambda-1$ is a square if and only if $(a,b)$ has a parent. Considering the point $(x,y)=(\lambda,0)$, we see that
    \begin{equation*}
        x-0,\quad x-1,\quad x-\lambda=\lambda,\quad \lambda-1,\quad 0,
    \end{equation*}
    are all squares if and only if $(a,b)$ has a parent, so by Lemma \ref{lem:2descent} $(\lambda,0)\in 2E_\lambda(\FF_q)$ if and only if $(a,b)$ has a parent. This implies $E_\lambda [4]\subseteq E_\lambda (\FF_q)$ if and only if $(a,b)$ has a parent.

    (2) Now suppose $(a,b)$ does not have a parent. Let $P$ be a point such that $2P=(0,0)$, so by Lemma \ref{lem:order4points}, we have that the $x$ coordinate of $P$ is $\pm \sqrt{\lambda}$. Applying $2$-descent to $P=(b/a,y_0)$ and $P'=(-b/a,y_0')$, we have
    \begin{align*}
        \frac ba,\frac ba-1,\frac ba-\lambda&\equiv ab,\frac ba-1,ab\left(1-\frac ba\right) \mod \FF_q^{\times2}, \\
        -\frac ba,-\frac ba-1,\frac ba-\lambda&\equiv ab,-\frac ba-1,ab\left(-1-\frac ba\right) \mod \FF_q^{\times2}.
    \end{align*}
    If $(a,b)$ does not have a child, then $ab$ and $-ab$ are not squares, completing one direction. Now assume $(a,b)$ has a child, then $ab$ and $-ab$ are squares. We compute,
    \begin{equation*}
        \left(\frac ba-1\right)\left(-\frac ba-1\right)=1-\lambda\equiv a^2-b^2 \mod \FF_q^{\times 2}.
    \end{equation*}
    Since $(a,b)$ does not have a parent by assumption, we know that $a^2-b^2$ is not square, so exactly one of $b/a-1$ and $-b/a-1$ is square. Suppose WLOG that $b/a-1$ is square. Then since $-ab$ is square, $ab(1-b/a)$ is also square, so $P\in 2E(\FF_q)$ by $2$-descent and $\ZZ/8\ZZ\oplus \ZZ/2\ZZ\subseteq E(\FF_q)$. This then implies that $\nu_2(|E(\FF_q)|)\geq 4$, completing the proof. 
\end{proof}

\begin{lemma}\label{lem:fishdisjoint}
    Let $q\equiv 5\mod 8$. Then if $(a,b)$ is in a fish (see Definition \ref{def:fish}) and $\lambda=\lambda(a,b)$, then $\nu_2(|E_\lambda(\FF_q)|)=4$. 
\end{lemma}

\begin{proof}
    Note that since $(a,b)$ is in a fish, $E_\lambda$ is isogenous to a elliptic curve coming from a pair $(a',b')$ with a parent but no children, so we may assume without loss of generality that $(a,b)$ has a parent but no children.

    As a consequence of Lemma \ref{lem:planktondisjoint}(1), we see that $E_\lambda[4]\subseteq E_\lambda(\FF_q)$. Therefore, to show that $\nu_2(|E(\FF_q)|)=4$, it suffices to show that no point of exact order $4$ is in $2E(\FF_q)$. Moreover, if $P,P'$ are points of exact order $4$ with $2P=2P'$, then $P-P'\in E[2]\subseteq 2E(\FF_q)$ so in this case $P\in 2E(\FF_q)$ if and only if $P'\in 2E(\FF_q)$. This means that instead of checking all $6$ $x$-coordinates of points of exact order $4$, we only need to check $3$ of them. Using Lemma \ref{lem:order4points} it suffices to show
    \begin{equation*}
        P_1=(\sqrt{\lambda},y_1) \qquad P_2=(1+\sqrt{1-\lambda},y_2) \qquad P_3=(\lambda+\sqrt{\lambda^2-\lambda},y_3)
    \end{equation*}
    are not in $2E(\FF_q)$.

    We first have that $P_1\not \in 2E(\FF_q)$ by Lemma \ref{lem:planktondisjoint}(2) since $2P_1=(0,0)$. For $P_2$, we consider $2$-descent applied to $P_2$:
    \begin{equation*}
        P_2:x,\quad x-1,\quad x-\lambda=1+\sqrt{1-\lambda}, \quad \sqrt{1-\lambda}, \quad 1-\lambda+\sqrt{1-\lambda}.
    \end{equation*}
    Since $(a,b)$ is the bottom of a fish, it has a parent. By Lemma \ref{lem:sqrt1minlambda} (1), this means there is a pair $(c,d)$ with $\lambda(c,d)=1-\lambda(a,b)$. But since $(a,b)$ has no grandparent, we get by Lemma \ref{lem:sqrt1minlambda} (2) that $(c,d)$ does not have a child, so $\sqrt{1-\lambda}=d/c$ is not square, and $P_2\not\in 2E(\FF_q)$.

    For $P_3$, we once again consider $2$-descent of $(\lambda+\sqrt{\lambda^2-\lambda},y_3)$:
    \begin{equation*}
        P_3:x,\quad x-1\quad x-\lambda =\lambda+\sqrt{\lambda^2-\lambda}, \quad \lambda-1+\sqrt{\lambda^2-\lambda}, \quad \sqrt{\lambda^2-\lambda}.
    \end{equation*}
    Consider $\sqrt{\lambda^2-\lambda}$, we see that
    \begin{equation*}
        \sqrt{\lambda^2-\lambda}\equiv \pm \sqrt\lambda\sqrt{\lambda-1}\equiv \pm i \sqrt{\lambda}\sqrt{1-\lambda} \mod \FF_q^{\times 2}.
    \end{equation*}
    By the cases of $P_1$ and $P_2$, we know that $\sqrt{\lambda}$ and $\sqrt{1-\lambda}$ are not squares, so since $i$ is not a square in $\FF_q$ (recall that $q\equiv 5\mod 8$) we get that $\sqrt{\lambda^2-\lambda}$ is not square. By $2$-descent, we obtain that $P_3\not\in 2E(\FF_q)$.
\end{proof}

These lemmas suffice to handle the $q\equiv 5\mod 8$ case of Theorem \ref{countingcomponents}, for the $q\equiv 1\mod 8$ case, we will need an additional lemma.

\begin{lemma}\label{lem:inhead}
    Let $E_\lambda$ be an elliptic curve in Legendre normal form over $\mathbb{F}_q$ with complex multiplication by an order $\mathcal{O}$ with odd index in the maximal order $\mathcal{O}_K$, and let $\left( \frac{{\rm disc}(K)}{2}\right) = 1$. Further, assume $E_\lambda(\mathbb{F}_q)[4] \cong \Z/4\Z \oplus \Z/4\Z$. Then there is a point $P$ in $V(\mathbb{F}_q)$ such that $E_{\lambda(P)} \cong E_\lambda$, and $P$ is part of a cycle in $A(\FF_q)$. 
\end{lemma}

\begin{proof}
    By Lemma \ref{lem:eventualhead}, there is a point $P \in V(\mathbb{F}_q) \subseteq  V(\mathbb{F}_{q^m})$ such that $P$ is part of a jellyfish head $H$ in $A(\mathbb{F}_{q^m})$, for some $m \geq 1$. We will show that all of the points in $H$ are also in $V(\mathbb{F}_{q})$, so that $P$ is part of a cycle in $A(\mathbb{F}_{q})$.  

    By Lemma \ref{lem:planktondisjoint}(1), $E_{\lambda(P)}(\mathbb{F}_q)[4] \cong \Z/4\Z \oplus \Z/4\Z$ implies that the parents of $P$ exist over $V(\mathbb{F}_{q})$. One of these must be in $H$. Call it $P_0$. Because $E_{\lambda(P_0)}$ has a horizontal isogeny to $E_{\lambda(P)}$, it must have CM by the same order, and so $E_{\lambda(P)}(\FF_q)$ and $E_{\lambda(P_0)}(\FF_q)$ are isomorphic as abelian groups by \cite[Thm. 10]{kohel1996endomorphism} and $E_{\lambda(P_0)}(\mathbb{F}_q)[4] \cong \Z/4\Z \oplus \Z/4\Z$. Then the parents of $P_0$ are \textit{also} in $V(\mathbb{F}_{q})$, and the point that precedes $P_0$ in the cycle $H$ must also be in $V(\mathbb{F}_{q})$. Proceeding this way, we see that every point in the cycle $H$ is in $V(\mathbb{F}_{q})$, and $P$ is part of a jellyfish head over $A(\mathbb{F}_{q})$.  
\end{proof}

We these lemmas established, we will can prove the main theorem of this section. What we will do is produce enough nonisogeneous elliptic curves that are guaranteed to be in a jellyfish by Lemmas \ref{lem:planktondisjoint}, \ref{lem:fishdisjoint}, and \ref{lem:inhead}. Since all elliptic curves in the same jellyfish are isogenous, this will lower bound the number of jellyfish.

\begin{proof}[Proof of Theorem \ref{countingcomponents}]
    We begin with the case of $q\equiv 5\mod 8$. By Lemmas \ref{lem:planktondisjoint} and \ref{lem:fishdisjoint}, we know that the isogeny classes associated to an $|E(\FF_q)|$ with $\nu_2(|E(\FF_q)|)>4$ only appear in connected components associated to jellyfish. There are asymptotically $\sqrt{q}(p-1)/(8p)$ such isogeny classes of elliptic curves over $\FF_q$ by \cite[Thm. 4.2]{ShoofCurveCounting}. Theorem \ref{prop:hurwitzclassnum}(1) tells us that all such isogeny classes appear at least once in $A(\FF_q)$, so there are at least $((p-1)/(8p)-\varepsilon)\sqrt{q}$ jellyfish asymptotically because elliptic curves in the same jellyfish are isogenous.
    
    Now we consider $q\equiv 1\mod 8$. Since elliptic curves in the same jellyfish are isogenous, it suffices to bound the number of isogeny classes of elliptic curves that are contained in jellyfish. In particular, we will determine integers $s\in [-2\sqrt q,2\sqrt q]$ such that there exists an elliptic curve $E/\FF_q$ with trace of Frobenius $s$ and such that there is a pair $(a,b)$ in a head with $E_{\lambda(a,b)}\cong_{\FF_q} E$.

    Let $\sqrt{q}$ be a $2$-adic square root of $q$, where if $q\equiv 1\mod 16$ we pick $\sqrt{q}\equiv 1\mod 8$ and if $q\equiv 9\mod 16$ we pick $\sqrt{q}\equiv 3\mod 8$. Let $s\equiv 46\sqrt{q}\mod 128$ and $p\nmid s$. We verify the following facts:
    \begin{enumerate}
        \item $s\equiv q+1\mod 16$,
        \item $16\mid s^2-4q$,

    \end{enumerate}
    by direct computation. By \cite[Thm. 4.2]{waterhouse1969abelian}, there exist $E/\FF_q$ with CM by $\mathcal{O}_{\QQ(\pi)}$ and trace of Frobenius $s$. Then by \cite[Thm. 3.7]{ShoofCurveCounting}, since $4^2\mid s^2-4q$, $4\mid q-1$, and $16\mid q+1-s$, we obtain $E[4]\subseteq E(\FF_q)$. Finally, by \ref{lem:hurwitzclassnum4tors}(1), $E\cong E_{\lambda}$ over $\FF_q$ for some $\lambda\in\FF_q^{\times 2}$. By Lemma \ref{lem:2splitsinendomorphismring}, we have that $\left(\frac{\disc(\QQ(\pi))}2\right)=1$. Finally, by Lemma \ref{lem:inhead} this implies there is an elliptic curve isomorphic to $E\cong E_\lambda$ in a head.

    The proportion of $s\in \ZZ$ that satisfy $s\equiv 46\sqrt{q}\mod 128$ and $p\nmid s$ is $(p-1)/(128p)$, so there are approximately $4(p-1)/(128p)\sqrt{q}$ such $s$ in $[-2\sqrt{q},2\sqrt{q}]$ by \cite[Thm. 4.2]{ShoofCurveCounting}, so for large $q$ we obtain
    \begin{equation*}
        d(\FF_q)>\left(\frac{p-1}{32p}-\varepsilon\right)\sqrt{q}, 
    \end{equation*}
    as desired.
\end{proof}

\section{Conclusions} \label{sec:conclusion}
In this paper, we explored the AGM over finite fields, generalizing several prior results in the $q \equiv 3 \pmod 4$ case to arbitrary finite fields and establishing new results on the type and structure of components, including a complete classification of components.  

This extends prior results on the AGM over finite fields, as well as more broadly results about isogenies of elliptic curves defined by the AGM. It also suggests an algorithm for fast computation of 2-isogenies between elliptic curves in Legendre normal form, via computation of the corresponding arithmetic-geometric mean function. 
Further work might explore the AGM over the $p$-adic integers.

\section*{Acknowledgements}
This research was supported by NSA MSP grant H98230-24-1-0033. The authors would like to thank Dr. Wei-Lun Tsai for his wonderful talk on \cite{GOSTJellyfish} at Clemson University which motivated this work, and his helpful conversations throughout our work on this paper.

\bibliographystyle{plain}

\bibliography{ref}

\end{document}